\documentclass[11pt,reqno,a4paper]{amsart}
 \usepackage{amsfonts}
\usepackage{epsfig}
\usepackage{graphicx}
\usepackage{amsmath}
\usepackage{amssymb}
\usepackage{colordvi}
\usepackage{times,amsmath,epsfig,float,multicol,subfigure}

\usepackage{amsfonts, amssymb, amsmath, amscd, amsthm, txfonts, color, enumerate}

 \usepackage{graphicx}
\usepackage{hyperref}

 \allowdisplaybreaks

\numberwithin{equation}{section}
\numberwithin{figure}{section}
\theoremstyle{plain}
\newtheorem{theorem}{Theorem}[section]
\newtheorem{proposition}[theorem]{Proposition}
\newtheorem{lemma}[theorem]{Lemma}
\newtheorem{corollary}[theorem]{Corollary}

\theoremstyle{definition}
\newtheorem{remark}[theorem]{Remark}
\newtheorem{example}[theorem]{Example}
\newtheorem{definition}[theorem]{Definition}

\numberwithin{equation}{section}

\title[Genericity of continuous maps with positive metric mean dimension]{Genericity of continuous maps with positive metric mean dimension}
\author{Jeovanny  M.  Acevedo}

\address{Jeovanny de Jesus Muentes Acevedo, Facultad de Ciencias B\'asicas,  Universidad Tecnol\'ogica de  Bol\'ivar, Cartagena de Indias - Colombia}
\email{jmuentes@utb.edu.co}

\begin{document}

\begin{abstract}
M. Gromov   introduced the  mean dimension for a continuous map in   the late 1990's,  which is  an invariant under topological conjugacy.  On the other hand, the notion of metric mean dimension for a dynamical system was   introduced by Lindenstrauss and Weiss in 2000 and this refines   the topological entropy for dynamical systems with infinite topological  entropy.   In  this paper we    will show if $N$ is a $n$ dimensional   compact riemannian manifold  then, for any $a\in [0,n]$, the set consisting of continuous maps with metric mean dimension equal to $a$ is dense in $C^{0}(N)$ and for $a=n$ this set is residual. Furthermore, we prove  some results related to existence and  density of  continuous maps   on   Cantor sets with positive metric mean dimension and on product spaces with positive mean dimension.  
\end{abstract}



\date{\today}
\maketitle

\section{Introduction}

Let $X$ be a compact metric space with metric $d$. The notion of mean dimension for a topological dynamical system $(X,\phi)$, which will be denoted by $\text{mdim}(X,\phi)$, was introduced by M. Gromov in \cite{Gromov}. It is another  invariant under topological conjugacy.  Applications and properties  of the mean dimension   can be found in \cite{GTM}, \cite{Gutman}, \cite{lind2}, \cite{lind}, \cite{lind3}, \cite{lind4}.  

\medskip 

Lindestrauss and Weiss in  \cite{lind},  introduced the notion of  metric mean dimension for any continuous map $\phi$ on $X$. This notion depends on the metric $d$ on $X$ (consequently it is not invariant under topological conjugacy) and it is zero for any map with finite topological entropy (see   \cite{lind}, \cite{lind3},  \cite{MeandFagner}). Some well-known properties of the topological entropy are valid for both the mean dimension and the metric mean dimension.  
We will study the veracity of other fundamental and topological  properties for the mean metric dimension.

\medskip
 
  In the next section we will present the definitions of     lower  metric mean dimension  and  upper metric mean dimension  of a dynamical system $(X,d,\phi)$, which will be denoted by $ 
 \underline{\text{mdim}}_{\text{M}}(X,d,\phi)$ and $\overline{\text{mdim}}_{\text{M}}(X,d,\phi)$, respectively. The definition of the mean dimension  $\text{mdim}(X,\phi)$ can be found in \cite{lind}.
 
 \medskip
 
 In Section \ref{section3} we will show the following properties of the metric mean dimension:
 \begin{itemize} \item It is well-known the metric mean dimension is not an invariant under topological conjugacy. In    Remark \ref{obscont} we will present an   example of a path of topologically conjugate  continuous maps with  different metric mean dimension.
  \item     Misiurewicz in \cite{Misiurewicz} proved if $\phi$ has a  $s$-horseshoe with $s\geq 2$, then $h_{\text{top}}(\phi)\geq \log s$. In Theorem \ref{misiu} we   present  a formula   for the metric mean dimension   related to the presence of horseshoes for a certain class of continuous maps on the interval. This formula allows us to provide an expression  for the metric mean dimension of the compositions of a continuous map (see Corollary  \ref{corollad}). 
  \item If    $\phi:X\rightarrow X$ and $\psi:Y\rightarrow Y$ are continuous maps ($Y$ is a metric space with metric $d^{\prime}$) we have $h_{\text{top}}(\phi\times \psi)= h_{\text{top}}(\phi)+h_{\text{top}}( \psi)$. This equality is not always valid for the (metric) mean dimension (see    Example \ref{mnbad}).    For the mean dimension we have $\text{mdim}(X\times Y, \phi\times \psi)\leq \text{mdim}(X,\phi)+\text{mdim}(Y ,\psi)$ (see \cite{lind}, Proposition 2.8).  This  inequality can be strong (see \cite{Tsukamoto}, Example 1.2).  In Theorem \ref{inequalitiess} we will present  lower and upper bounds for  {both}  $\overline{\text{mdim}}_{\text{M}}(X\times Y,d\times d^{\prime}, \phi\times \psi )$ and $\underline{\text{mdim}}_{\text{M}}(X\times Y,d\times d^{\prime}, \phi\times \psi )$.  \end{itemize}

   \medskip
   
Let $N$ be a compact riemannian  manifold with $n=\dim(N)$. Yano in \cite{Yano} showed if $n\geq 2$, then  set consisting of homeomophisms on $N$ whose topological entropy is infinite is a residual subset of $\text{Hom}(N)$.   In  \cite{Carvalho},  the authors proved  if  $n\geq 2$, then the set consisting  of  homeomorphisms with upper metric mean dimension equal to $ n$ is residual in $ \text{Hom}(N)$.    In Section \ref{section6}  we will show for any $a\in [0,n]$ the set consisting of continuous maps with lower and upper metric mean dimension equal to $a$ is dense in $C^{0}(N)$ (see Theorems \ref{densidadd} and  \ref{densitypositivemanifold}).  Furthermore, the set consisting of continuous maps with   upper metric mean dimension equal to $n$ is residual (see Theorem \ref{teoresidual}).   From these results we have   the metric mean dimension map is not continuous anywhere on the set consisting of continuous maps defined on manifolds  (see Corollaries \ref{continuitymandimension} and  \ref{hfjdjehr}). 

\medskip

In Section   \ref{section4}  we will show the existence of continuous maps on Cantor sets with positive metric mean dimension (see Proposition \ref{gshfkf}).     Bobok and Zindulka in \cite{Bobok} shown the existence of homeomorphisms, defined  on  uncountable compact metrizable spaces with topological dimension zero, with infinite topological entropy.   We will use these  techniques  in order to prove    there exist   continuous maps  on the Cantor set with positive metric mean dimension (see Proposition \ref{gshfkf}) and furthermore the density of these maps (see  Theorem \ref{bcbcbcb123}).    Block, in \cite{block}, proved the topological entropy map is not continuous anywhere on the set consisting of continuous map on  Cantor sets. This fact also holds for metric mean dimension map (see Theorem \ref{bcbcbcb}).    We will finish this work showing some results related to the density of continuous maps, defined on product spaces, with positive mean dimension (see Theorem \ref{mendimen}).

\section{Mean dimension and metric mean dimension}\label{section2}

Let  $\alpha $ be a finite  open cover of a compact topological space $X$. Set 
\[\text{ord}(\alpha)=\sup_{x\in X}\sum_{U\in \alpha }1_{U}(x)-1\quad \quad \text{ and }\quad \quad \mathcal D(\alpha )=\min_{\beta\succ\alpha}\text{ord}(\beta),\]
                   where $1_{U}$ is the indicator function and  $\beta \succ \alpha$ means that $\beta$ is a finite open cover  of $X$ finer than $\alpha$. Recall that for a topological space $X$, the \textit{topological dimension} is defined as
$$
\text{dim}(X)=\sup_{\alpha}\mathcal{D}(\alpha),
$$
where $\alpha $ runs over all finite  open covers of $X$.  For any continuous map $\phi: X\rightarrow X$, define
\[
\alpha_{0}^{n-1}=\alpha \vee (\phi^{-1}(\alpha))\vee (\phi^{-2}(\alpha))\vee \dots \vee (\phi^{-n+1}(\alpha)).
\]

\begin{definition} The \textit{mean dimension} of $\phi:X\rightarrow X$ is defined to be
\begin{equation*}
\text{mdim}(X,\phi)=\sup_{\alpha}\lim_{n\to\infty}\frac{\mathcal D(\alpha_{0}^{n-1})}{n},
\end{equation*}
where $\alpha$ runs over all finite open covers of X.
 \end{definition}

   If $\text{dim}(X)<\infty$, where   $\text{dim}(X)$ is the topological dimension of $X$,  then $ \text{mdim}(X,\phi)=0 $ (see \cite{lind}). Furthermore, in \cite{lind}, Proposition 3.1, is proved that  $\text{mdim}(X^{\mathbb{Z}},\sigma)\leq \text{dim}(X)$, where $\sigma$ is the shift map on $X^{\mathbb{Z}}$. 
   
   \medskip
   
   Let  $X$ be  a compact metric space endowed with a metric $d$ and $\phi:X\rightarrow X$ a continuous map.   For any non-negative integer
$n$ we define $d_n:X\times X\to [0,\infty)$ by
$$
d_n(x,y)=\max\{d(x,y),d(\phi(x),\phi(y)),\dots,d(\phi^{n-1}(x),\phi^{n-1}(y))\}.
$$   Fix $\varepsilon>0$. We say that $A\subset X$ is {an} $(n,\phi,\varepsilon)$-\textit{separated} set
if $d_n(x,y)>\varepsilon$, for any two  distinct points  $x,y\in A$. We denote by $\text{sep}(n,\phi,\varepsilon)$ the maximal cardinality of {an} $(n,\phi,\varepsilon)$-separated
subset of $X$.   We say that $E\subset X$ is {an} $(n,\phi,\varepsilon)$-\textit{spanning} set for $X$ if
for any $x\in X$ there exists $y\in E$ such  that $d_n(x,y)<\varepsilon$. Let $\text{span}(n,\phi,\varepsilon)$ be the minimum cardinality
of any $(n,\phi,\varepsilon)$-{spanning} subset of $X$.   Given an open cover $\alpha$, we say that $\alpha$ is {an}
$(n,\phi,\varepsilon)$-\textit{cover} of $X$ if the $d_n$-diameter of any element of $\alpha$ is less than
 $\varepsilon$. Let  $\text{cov}(n,\phi,\varepsilon)$ be the minimum number of elements in any $(n,\phi,\varepsilon)$-cover of $X$.    Set \begin{itemize} \item $\text{sep}(\phi,\varepsilon)=\underset{n\to\infty}\limsup \frac{1}{n}\log \text{sep}(n,\phi,\varepsilon)$;
\item $\text{span}(\phi,\varepsilon)=\underset{n\to\infty}\limsup \frac{1}{n}\log \text{span}(n,\phi,\varepsilon)$;
\item  $\text{cov}(\phi,\varepsilon)=\underset{n\to\infty}\limsup \frac{1}{n}\log \text{cov}(n,\phi,\varepsilon)$.\end{itemize}

 \begin{definition}
  The \emph{topological entropy} of $(X,\phi,d)$   is defined by      
  \begin{equation*} h_{\text{top}}(\phi)=\lim _{\varepsilon\to0} \text{sep}(\phi,\varepsilon)=\lim_{\varepsilon\to0}  \text{span}(\phi,\varepsilon)=\lim_{\varepsilon\to0}\text{cov}(\phi,\varepsilon).
\end{equation*}
  \end{definition}

 \begin{definition}
 The \emph{lower  metric mean dimension}    and the \emph{upper metric mean dimension} of $(X,d,\phi)$ are defined by
  \begin{align*}\label{metric-mean}
 \underline{\text{mdim}}_{\text{M}}(X,d,\phi)&=\liminf_{\varepsilon\to0} \frac{\text{sep}(\phi,\varepsilon)}{|\log \varepsilon|}=\liminf_{\varepsilon\to0} \frac{\text{span}(\phi,\varepsilon)}{|\log \varepsilon|}=\liminf_{\varepsilon\to0} \frac{\text{cov}(\phi,\varepsilon)}{|\log \varepsilon|} \\
 \overline{\text{mdim}}_{\text{M}}(X,d,\phi)&=\limsup_{\varepsilon\to0} \frac{\text{sep}(\phi,\varepsilon)}{|\log \varepsilon|}=\limsup_{\varepsilon\to0} \frac{\text{span}(\phi,\varepsilon)}{|\log \varepsilon|}=\limsup_{\varepsilon\to0} \frac{\text{cov}(\phi,\varepsilon)}{|\log \varepsilon|},
\end{align*}
respectively. \end{definition}

  \begin{remark}\label{erfer}
 Throughout the paper, we will omit the underline and the overline  on the notations $\underline{\text{mdim}}_{\text{M}}$ and $\overline{\text{mdim}}_{\text{M}}$      when the result be valid for both cases, that is,  we  will use $\text{mdim}_{\text{M}}$ for the both cases. 
 \end{remark}

\section{Some fundamental properties}\label{section3} 
   One of the most important properties of the topological entropy is that it is an invariant under topological conjugacy.      Mean dimension  is an invariant under topological conjugacy (see   \cite{lind}).   It is well-known    the metric mean dimension for  continuous maps  depends on the metric $d$ on $X$. Consequently, it is not an invariant under topological  conjugacy between   dynamical systems.  From the next example we will show that we can find paths of continuous maps that are topologically conjugate and have different metric mean dimension.

  \begin{example}\label{EXAMPLE1} Fix  $r\in(0,\infty)$. Set $a_{0}=0$ and $a_{n}= \sum_{i=0}^{n-1}\frac{C}{3^{ir}}$ for $n\geq 1$, where $C=\frac{1}{\sum_{i=0}^{\infty}\frac{1}{3^{ir}}}= \frac{3^{r}-1}{3^{r}}$. For each $n\geq 0$, let 
 $T_{n}: I_{n}:=[a_{n},a_{n+1}] \rightarrow [0,1] $ be the unique increasing affine map from $I_{n}$   onto $[0,1]$. \begin{figure}[ht]
 \centering
 { 
  \includegraphics[width=0.23\textwidth]{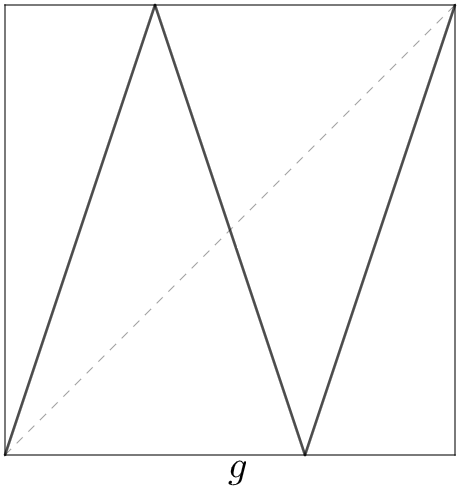}}  
 \quad{\label{fig:ejemplo12}
 \includegraphics[width=0.23\textwidth]{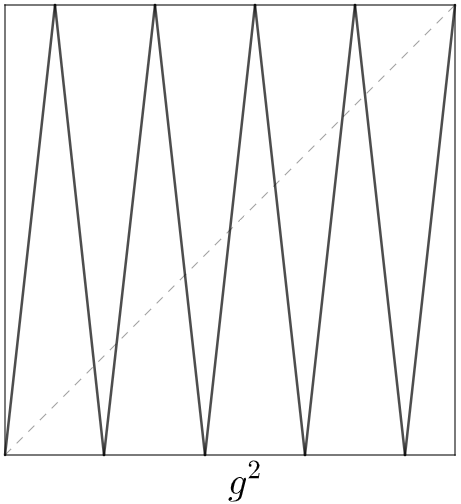}}
\quad{ \label{1fig:ejemplo142} \includegraphics[width=0.23\textwidth]{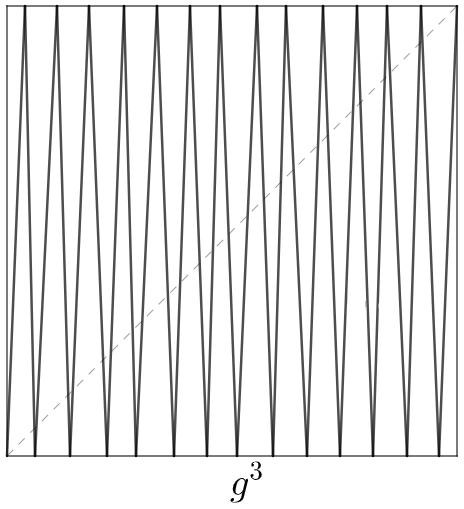} }
\caption{Graphs of $g$, $g^{2}$, $g^{3}$} \label{fig:ejemplo1}
\end{figure}
For $s\in \mathbb{N}$, set $\phi_{s,r}:[0,1]\rightarrow [0,1]$, given by   $\phi_{s,r}|_{I_{n}}= T_{n}^{-1}\circ g^{s(n+1)}\circ T_{n}$ for any $n\geq 0$,  where $g:[0,1]\rightarrow [0,1]$, is defined by $x\mapsto |1-|3x-1||$ (see Figure \ref{fig:ejemplo1}).  We will prove that $$ {{\text{mdim}}_{\text{M}}}([0,1] ,|\cdot |,\phi_{s,r}) =\frac{s}{r+s}\quad\text{for any }s\in \mathbb{N}. $$
Take any $\varepsilon\in (0,1)$. For any $k\geq 1$ set $ \varepsilon_k=  \frac{|I_{k}|}{3^{s (k+1)}}  =\frac{C}{3^{k(r+s)+s}} ,$ where $|I_{k}|=a_{k+1}-a_{k}$. There exists  some $k\geq 1$ such that  $\varepsilon \in [\varepsilon_{k}, \varepsilon_{k-1}]$. Note that  
 \begin{equation*}\label{equ12ssswe}   \text{sep}( n, \phi_{s,r}  , \varepsilon)\geq \text{sep}( n, \phi_{s,r}  , \varepsilon_{k-1})\geq   \text{sep}( n, \phi_{s,r} |_{I_{k-1}}, \varepsilon_{k-1}) \quad\text{for any }n\geq 1 .  \end{equation*}
 From Lemma 6 in \cite{VV} it follows that for any $n\geq 1$ we have   
 \begin{equation*}\label{equ12sss}     \text{sep}( n, \phi_{s,r} |_{I_{k-1}}, \varepsilon_{k-1}) \geq   \left(\frac{3^{sk}}{2}\right)^{n} \quad\text{and hence }\quad 
  \text{sep}(\phi_{s,r}, \varepsilon)\geq    \log \left(\frac{3^{sk}}{2}\right) .
 \end{equation*}
 Thus  \begin{align*}\label{exxample12} {\underline{\text{mdim}}_{\text{M}}}([0,1] ,|\cdot |,\phi_{s,r}) &= \liminf_{\varepsilon\rightarrow 0} \frac{\text{sep}(\phi_{s,r} ,\varepsilon )}{|\log  {\varepsilon}|} \geq \lim_{k\rightarrow \infty}\frac{ \log 3^{sk}}{|\log \varepsilon_{k}|}\\
 &=\lim_{k\rightarrow \infty}\frac{\log 3^{sk}}{\log 3^{k(r+s)+s}} =\frac{s}{r+s}.\end{align*}
 
 On the other hand, note that $\frac{s(k+1) \log 3}{((k-1)(r+s)+s)\log3-\log C}\rightarrow \frac{s}{r+s}$ as $k\rightarrow \infty$. Hence, for any $\delta >0$ there exists $k_{0}\geq 1$ such that for any  $k> k_{0}$ we have $\frac{s(k+1) \log 3}{((k-1)(r+s)+s)\log3-\log C}< \frac{s}{r+s}+\delta$.  
 Hence, suppose that $\varepsilon$ is small enough such that $\varepsilon<\varepsilon_{k_{0}-1}$. Let $k\geq k_{0}$ such that $\varepsilon\in[\varepsilon_{k},\varepsilon_{k-1}]$. For each $0\leq j\leq k$,   dividing  each $I_{j}$ into $\frac{3^{s(j+1)n}|I_{j}|}{\varepsilon}$ sub-intervals with the same length, we have   the set consisting of the end points  of these sub-intervals is an $(n,\phi_{s,r},\varepsilon)$-spanning set (see \cite{demelo}, Corollary 7.2). Hence, if $Y_{k}=\cup_{j=0}^{k}I_{j}$, for every $ n\geq1$ we have 
  \begin{align*}    \text{span}(n,\phi_{s,r} |_{Y_{k}}, \varepsilon) & \leq \sum_{j=0}^{k}   \frac{3^{s(j+1)n}|I_{j}|}{\varepsilon}  \leq \sum_{j=0}^{k}   \frac{3^{s(j+1)n}|I_{j}|}{\varepsilon_{k}} =  \sum_{j=0}^{k} \frac{3^{s(j+1)n}3^{s(k+1)}|I_{j}|}{|I_{k}|}  \\
  &= \sum_{j=0}^{k} \frac{3^{sn(j+1)}3^{s(k+1)}3^{kr}}{3^{jr}}\leq(k+1) 3^{s(k+1)n} 3^{s(k+1)+kr}.  \end{align*}
 Hence 
 \begin{align*} \frac{\text{span}(\phi_{s,r} |_{Y_{k}}  , \varepsilon)}{|\log\varepsilon|}& \leq\limsup_{n\rightarrow \infty} \frac{\log[ (k+1)     3^{s(k+1)n} 3^{s(k+1)+kr}]}{n|\log \varepsilon_{k-1}|}\\
 & =\limsup_{n\rightarrow \infty} \frac{s(k+1)\log 3}{[((k-1)(r+s)+s)\log3-\log C ]}\\
 &= \frac{s(k+1) \log 3}{((k-1)(r+s)+s)\log3-\log C}<\frac{s}{s+r}+\delta.
 \end{align*}
 This fact implies that for any $\delta >0$ we have \begin{align*}  {\overline{\text{mdim}}_{\text{M}}}([0,1] ,|\cdot |,\phi_{s,r})<\frac{s}{r+s}+\delta\quad\text{and hence }\quad {\overline{\text{mdim}}_{\text{M}}}([0,1] ,|\cdot |,\phi_{s,r})\leq\frac{s}{r+s}\end{align*}
  The above facts proves ${ {\text{mdim}}_{\text{M}}}([0,1] ,|\cdot |,\phi_{s,r})=\frac{s}{r+s}. $
  \end{example}

Note for each $s\geq 1$ and $r\in (0,1) $ we have \begin{equation}\label{sdfwddds} \phi_{s,r}=\phi_{1,r}^{s}.   \end{equation} 
Hence, in this case we have
$${ {\text{mdim}}_{\text{M}}}([0,1] ,|\cdot |,\phi_{1,r}^{s})=\frac{s}{r+s}= \frac{s\, { {\text{mdim}}_{\text{M}}}([0,1] ,|\cdot |,\phi_{1,r} )}{{ {\text{mdim}}_{\text{M}}}([0,1] ,|\cdot |,\phi_{1,r})(s+1)-1}\quad\text{for each }s\in\mathbb{N}.$$

\begin{remark}\label{obscont}    Let $r_{1}>0$ and $r_{2}>0$. For each $n\geq 1$, take $I^{r_{1}}_{n}$ and $I^{r_{2}}_{n}$ the intervals obtained as in the Example \ref{EXAMPLE1}, for $r_{1}$ and $r_{2}$, respectively. Fix $s\geq 1 $ and let $\phi_{s,r_{1}}$ and $\phi_{s,r_{2}}$ be the continuous maps defined above for $r_{1} $ and $r_{2} $, respectively.  
 Note that, for each $n\geq 0$,  $ \phi_{s,r_{1}}|_{I_{n}^{r_{1}}}$ and $ \phi_{s,r_{2}}|_{I_{n}^{r_{2}}}$ are topologically conjugate by a continuous map $h_{n}:I_{n}^{r_{1}}\rightarrow I_{n}^{r_{2}}$: 
 $$ \phi_{s,r_{1}}|_{I_{n}^{r_{1}}}=h_{n}^{-1}\circ \phi_{s,r_{2}}|_{I_{n}^{r_{2}}} \circ h_{n}  . $$
 Therefore, $\phi_{s,r_{1}}$ and $\phi_{s,r_{2}}$  are topologically conjugate by $h:I\rightarrow I$ given by $h|_{I_{n}^{r_{1}}}=h_{n}$ for each $n\geq 0$. This fact proves, for each $s\in\mathbb{N}$,  $\mathcal{A}_{s}=\{ \phi_{s,r}: r\in (0,\infty)\}$ is a path of topologically conjugate continuous maps such that $\text{mdim}_{\text{M}}([0,1],|\cdot |,\phi_{s,r})=\frac{s}{r+s}$ for each   $r\in (0,\infty)$.     \end{remark}

 A $s$-\textit{horseshoe}   for $\phi:[0,1]\rightarrow [0,1]$ is an interval $J\subseteq [0,1]$ which has a partition into $s$ subintervals $J_{1},\dots,J_{s}$ such that $J_{j}^{\circ}\cap  J_{i}^{\circ}=\emptyset$ for $i\neq j$ and $J\subseteq \phi (\overline{J}_{i})$ for each $i=1,\dots, s$.
 
 \medskip
 
 If $g$ is the map defined in      Example \ref{EXAMPLE1}, we have $I=[0,1]$ is an $3$-horseshoe for $g$. Furthermore, for $n\geq 0$,
 each  $I_{n}$ can be divided into $3^{s(n+1)}$ closed intervals with the same length  $I_{n}^{1},\dots,$ $  I_{n}^{3^{s(n+1)}}$, such that 
\[ \phi_{s,r}(I_{n}^{i})=I_{n} \quad\text{ for each } i\in\{1,\dots ,3^{s(n+1)}\}.\]     Consequently, each $I_{n}$ is a $3^{s(n+1)}$-horseshoe for $\phi_{s,r}$. 

\medskip

Misiurewicz in \cite{Misiurewicz}, proved if $\phi$ has a  $s$-horseshoe with $s\geq 2$, then $h_{\text{top}}(\phi)\geq \log s$.   Suppose  for each $k\in\mathbb{N}$ there exists    a   $s_{k}$-horseshoe for $\phi\in C^{0}([0,1])$,  $I_{k}=[a_{k-1},a_{k}]\subseteq [0,1]$, consisting of sub-intervals    $I_{k}^{1}, I_{k}^{2},\dots, I_{k}^{s_{k}} $ with the same length, where $s_{k}\geq 2$ for all $k\geq 1$. From   Lemma 6 in  \cite{VV} we can prove that (see Example \ref{equ12ssswe})  \begin{equation}\label{eqas1} \overline{\text{mdim}}_{\text{M}}([0,1] ,|\cdot |,\phi) \geq  \underset{k\rightarrow \infty}{\limsup}\frac{ 1}{\left|1-\frac{\log |I_{k}|}{\log s_{k}}\right|} . \end{equation}   

Next  theorem provides  upper bounds for the lower metric mean dimension. 
 
 \begin{theorem}\label{misiu}
Suppose  for each $k\in\mathbb{N}$ there exists    a   $s_{k}$-horseshoe for $\phi\in C^{0}([0,1])$,  $I_{k}=[a_{k-1},a_{k}]\subseteq [0,1]$, consisting of sub-intervals   with the same length  $I_{k}^{1}, I_{k}^{2},\dots, I_{k}^{s_{k}} $ and $[0,1]=\cup_{k=1}^{\infty}I_{k}$. 
We can rearrange the intervals and suppose that  $2\leq s_{k}\leq s_{k+1}$ for each $k$. If each $\phi|_{I_{k}^{i}}:I_{k}^{i}\rightarrow I_{k} $ is a bijective  affine  map for all $k$ and $i=1,\dots, s_{k}$, we have \begin{enumerate}[i.] \item  $   {\underline{\emph{mdim}}_{\emph{M}}}([0,1] ,|\cdot |,\phi) \leq  \underset{k\rightarrow \infty}{\liminf}\frac{1}{\left|1-\frac{\log |I_{k}|}{\log s_{k}}\right|}. $  
\item If the limit $\underset{k\rightarrow \infty}{\lim}\frac{1}{\left|1-\frac{\log |I_{k}|}{\log s_{k}}\right|}$ exists, then $\overline{\emph{mdim}}_{\emph{M}}([0,1] ,|\cdot |,\phi) =  \underset{k\rightarrow \infty}{\lim}\frac{1}{\left|1-\frac{\log |I_{k}|}{\log s_{k}}\right|}.$ \end{enumerate}
 \end{theorem}
 
 \begin{proof}    Let $k_{i}$ be a strictly increasing sequence of positive integers such that  $$a:= \underset{k\rightarrow \infty}{\liminf}\frac{1}{\left|1-\frac{\log |I_{k}|}{\log s_{k}}\right|}= \underset{i\rightarrow \infty}{\lim}\frac{1}{\left|1-\frac{\log |I_{k_{i}}|}{\log s_{k_{i}}}\right|}.$$      For any $\delta >0$,  there exists $k_{0}$ such that  if, $k_{i}\geq k_{0}$, then $\frac{1}{\left|1-\frac{\log |I_{k_{i}}|}{\log s_{k_{i}}}\right|}<a+\delta$. For any $k_{i}\geq k_{0}$, set     $\varepsilon_{k_{i}}=  \frac{|I_{k_{i}}|}{s_{k_{i}}}$.  
 For each $1\leq j\leq k_{i}$,  dividing  each $I_{j}$ into $\frac{s_{j}^{n}|I_{j}|}{\varepsilon_{k_{i}}}$ sub-intervals with the same length, we have   the set consisting of the end points of these sub-intervals is an $(n,\phi|_{I_{j}} ,\varepsilon_{k_{i}})$-spanning set (see \cite{demelo}, Corollary 7.2). Hence, if $Y_{k_{i}}=\cup_{j=1}^{k_{i}}I_{j}$, for every $ n\geq1$ we have 
  \begin{align*}    \text{span}( n, \phi|_{Y_{k_{i}}}, \varepsilon_{k_{i}}) & \leq  \sum_{j=1}^{k_{i}}  \frac{s_{j}^{n}|I_{j}|}{\varepsilon_{k_{i}}} \leq \sum_{j=1}^{k_{i}}   \frac{s_{k_{i}}^{n}|I_{j}|}{\varepsilon_{k_{i}}} .  \end{align*}
   Thus 
 \begin{align*} \frac{\text{span}(\phi |_{Y_{k_{i}}}, \varepsilon_{k_{i}})}{|\log  {\varepsilon_{k_{i}}}|}& \leq\limsup_{n\rightarrow \infty} \frac{\log\left[\sum_{j=1}^{k_{i}}   \frac{s_{k_{i}}^{n}|I_{j}|}{\varepsilon_{k_{i}}} \right]}{n |\log s_{k_{i}}- \log |I_{k_{i}}| |} = \frac{1}{\left|1-\frac{\log |I_{k_{i}}|}{\log s_{k_{i}}}\right|}<a+\delta    
 \end{align*}
 This fact implies that for any $\delta >0$ we have \begin{align*}\label{exxample1} {\underline{\text{mdim}}_{\text{M}}}([0,1] ,|\cdot |,\phi ) &\leq a +\delta\quad\text{and hence}\quad {\underline{\text{mdim}}_{\text{M}}}([0,1] ,|\cdot |,\phi ) \leq a,\end{align*}
 which proves i.

Next, we will prove ii. 
From \eqref{eqas1} we have  \begin{equation}\label{eqas1123}  {\underline{\text{mdim}}_{\text{M}}}([0,1] ,|\cdot |,\phi ) \leq   \underset{k\rightarrow \infty}{\lim}\frac{ 1}{\left|1-\frac{\log |I_{k}|}{\log s_{k}}\right|} \leq \overline{\text{mdim}}_{\text{M}}([0,1] ,|\cdot |,\phi) . \end{equation} 
We can prove that for any $\delta >0$ there exists $k_{0}$ such that for any $k\geq k_{0}$ we have \begin{align*} \frac{\text{span}(\phi |_{Y_{k}}, \varepsilon)}{|\log  {\varepsilon}|}& \leq \frac{\text{span}(\phi |_{Y_{k}}, \varepsilon)}{|\log  {\varepsilon_{k}}|} \leq \frac{\log s_{k}}{\left|\log s_{k}- {\log |I_{k}|} \right|}=    \frac{1}{\left| 1- \frac{\log |I_{k}|}{\log s_{k}} \right|} <a+\delta,    
 \end{align*}
 for any  $\varepsilon>0$  small enough such that $\varepsilon\leq \varepsilon_{k}$ (see Example \ref{EXAMPLE1}).  Hence $${\overline{\text{mdim}}_{\text{M}}}([0,1] ,|\cdot |,\phi )\leq \underset{k\rightarrow \infty} {\lim}\frac{1}{\left|1- \frac{\log |I_{k}|}{\log s_{k}}\right|}.$$
 The equality follows from \eqref{eqas1123}.
   \end{proof}
 
It is well-known that for any continuous map $\phi:X\rightarrow X$ and $s\in\mathbb{N}$ we have $$ {\text{mdim}}_{\text{M}}([0,1] ,|\cdot |,\phi^{s}) \leq s\,   {\text{mdim}}_{\text{M}}([0,1] ,|\cdot |,\phi )  $$ and this inequality can be strict. 
Next corollary, which follows directly  from   Theorem \ref{misiu},  provides a formula for the metric mean dimension of the compositions of a map satisfying the conditions of the theorem. 
 
 \begin{corollary}\label{corollad}
  If $\phi$ is a map  which satisfies the properties of Theorem \ref{misiu} then for any $s\in \mathbb{N}$ we have $$\overline{\emph{mdim}}_{\emph{M}}([0,1] ,|\cdot |,\phi^{s}) \geq  \underset{k\rightarrow \infty}{\limsup}\frac{s}{\left|s-\frac{\log |I_{k}|}{\log s_{k}}\right|}\,\text{ and }\, {\underline{\emph{mdim}}_{\emph{M}}}([0,1] ,|\cdot |,\phi^{s}) \leq  \underset{k\rightarrow \infty}{\liminf}\frac{s}{\left|s-\frac{\log |I_{k}|}{\log s_{k}}\right|}. $$
 If the limit $\underset{k\rightarrow \infty}{\lim}\frac{1}{\left|1-\frac{\log |I_{k}|}{\log s_{k}}\right|}$ exists, then $  \overline{\emph{mdim}}_{\emph{M}}([0,1] ,|\cdot |,\phi^{s}) =  \underset{k\rightarrow \infty}{\lim}\frac{s}{\left|s-\frac{\log |I_{k}|}{\log s_{k}}\right|}.$
 \end{corollary}
 
Note that in Example \ref{EXAMPLE1}, for each  map $\phi_{s,r}$ we have 
$$  \underset{k\rightarrow \infty}{\lim} \frac{\log |I_{k}|}{\log s_{k}} =- \underset{k\rightarrow \infty}{\lim} \frac{\log 3^{kr}}{\log 3^{s(k+1)}}=-\frac{r}{s}.$$

  \begin{example}\label{med1}   Set $a_{0}=0$ and $a_{n}= \sum_{i=1}^{n}\frac{6}{\pi^{2}i^{2}}$ for $n\geq 1$. Set  $I_{n}:=[a_{n-1},a_{n}]$ for any $n\geq 1$.    Let $\varphi\in C^{0}([0,1])$ be defined by  $\varphi |_{I_{n}} = T_{n}^{-1}\circ g^{n}\circ T_{n}$ for any $n\geq 1$, where $T_{n}$ and $g$ are as in  Example \ref{EXAMPLE1}   (see Example 3.4 in \cite{MeandFagner}).   For   $\varphi^{s}$, with $s\in\mathbb{N},$  we have   $s_{k}=3^{sk}$ for each $k\in\mathbb{N}$. Therefore,
 $$  \underset{k\rightarrow \infty}{\lim} \frac{\log |I_{k}|}{\log s_{k}} =- \underset{k\rightarrow \infty}{\lim} \frac{\log k^{2}}{\log 3^{sk}}=0\quad\text{for any }s\in\mathbb{N}.$$
 It is follows from    Theorem \ref{misiu} that  $$ {\overline{\text{mdim}}_{\text{M}}}([0,1] ,|\cdot |,\varphi^{s}) =   \underset{k\rightarrow \infty}{\lim}\frac{s}{\left|s-\frac{\log |I_{k}|}{\log s_{k}}\right|} =1\quad\text{for any }s\in \mathbb{N}. $$
 The equality $  \underline{\text{mdim}}_{\text{M}}([0,1] ,|\cdot |,\psi^{s})  =1$ can be proved as in Example \ref{EXAMPLE1}. 
  \end{example}
  
  \begin{example}\label{EXAMPLE134}    Take $I_{n}=[a_{n-1},a_{n}]$ as in the above example.   Divide each interval $I_{n}$ into $2n+1$ sub-intervals with the same lenght, $I_{n}^{1}$, $\dots$, $I_{n}^{2n+1}$. For $k=1,3,\dots, 2n+1$, let $\psi|_{I_{n}^{k}} :I_{n} ^{k}\rightarrow I_{n} $ be  the unique increasing affine map from $I_{n} ^{k}$ onto $I_{n}$ and for $k=2,4,\dots, 2{n}$, let $\psi|_{I_{n}^{k}} :I_{n} ^{k}\rightarrow I_{n} $ be  the unique decreasing affine map from $I_{n} ^{k}$ onto $I_{n}$.  For $\psi^{s}$, with $s\in\mathbb{N},$  we have $| I_{k}|= \frac{6}{\pi^{2} k^{2}}$ and  $s_{k}=(2k+1)^{s}$ for each $k\in\mathbb{N}$. Therefore, 
 $$  \underset{k\rightarrow \infty}{\lim} \frac{\log |I_{k}|}{\log s_{k}} =- \underset{k\rightarrow \infty}{\lim} \frac{\log k^{2}}{\log (2k+1)^{s}}=-\frac{2}{s}.$$
 It  follows from    Theorem \ref{misiu} that  $$ \overline{\text{mdim}}_{\text{M}}([0,1] ,|\cdot |,\psi^{s}) =   \underset{k\rightarrow \infty}{\lim}\frac{1}{\left|1-\frac{\log |I_{k}|}{\log s_{k}}\right|} =\frac{s}{s+2}\quad\text{for any }s\in \mathbb{N}. $$ The equality $  \underline{\text{mdim}}_{\text{M}}([0,1] ,|\cdot |,\psi^{s})  =\frac{s}{s+2}$ can be proved as in Example \ref{EXAMPLE1}.  
  \end{example}

  Take $\phi:X\rightarrow X$ and $\psi:Y\rightarrow Y$ where $Y$ is a compact metric space with metric $d^{\prime}$. On $X\times Y$ we consider the metric \begin{equation}\label{bnm}(d\times d^{\prime})((x_{1},y_{1}),(x_{2},y_{2}))=d(x_{1},x_{2}) + d^{\prime}(y_{1},y_{2}),\quad \text{ for  }x_{1},x_{2}\in X \text{  and }y_{1},y_{2}\in Y.\end{equation}  The map  $\phi\times \psi : X\times Y\rightarrow X\times Y$ is defined to be $(\phi\times \psi)(x,y)=(\phi(x),\psi(y))$ for any $(x,y)\in X\times Y$. The equality $h_{top}(\phi\times \psi)=h_{top}(\phi)+h_{top}(\psi)$ always hold. Lindenstrauss in \cite{lind}, Proposition 2.8, proved that   \begin{equation}\label{fisrtt}\text{mdim}(X\times Y, \phi\times \psi)\leq \text{mdim}(X,\phi)+\text{mdim}(Y ,\psi)\end{equation} and this  inequality can be strict (see    \cite{Jin} and \cite{Tsukamoto}). For the metric mean dimension we   have: 

\begin{theorem}\label{inequalitiess} Take two continuous maps $\phi:X\rightarrow X$ and $\psi:Y\rightarrow Y$. On $X\times Y$   consider the metric given in \eqref{bnm}.  We have:
\begin{enumerate}[i.] 
\item $ \overline{\emph{mdim}}_{\emph{M}}(X\times Y,d\times d^{\prime}, \phi\times \psi ) \leq \overline{\emph{mdim}}_{\emph{M}}(X,d,\phi)+\overline{\emph{mdim}}_{\emph{M}}(Y  ,{d^{\prime}},\psi).$
\item $  \overline{\emph{mdim}}_{\emph{M}}(X,d,\phi)+\underline{\emph{mdim}}_{\emph{M}}(Y  ,{d^{\prime}},\psi)\leq \overline{\emph{mdim}}_{\emph{M}}(X\times Y ,d\times d^{\prime}, \phi\times \psi).$
\item $ \underline{\emph{mdim}}_{\emph{M}}(X,d,\phi)+\underline{\emph{mdim}}_{\emph{M}}(Y  ,{d^{\prime}},\psi)\leq \underline{\emph{mdim}}_{\emph{M}}(X\times Y ,d\times d^{\prime},  \phi\times \psi).$
\item $  \underline{\emph{mdim}}_{\emph{M}}(X\times Y,d\times d^{\prime}, \phi\times \psi )\leq  \underline{\emph{mdim}}_{\emph{M}}(X,d,\phi)+\overline{\emph{mdim}}_{\emph{M}}(Y  ,{d^{\prime}},\psi).$
\item If $ \overline{\emph{mdim}}_{\emph{M}}(X ,d,\phi)=\underline{\emph{mdim}}_{\emph{M}}(X ,d,\phi)$  or $ \overline{\emph{mdim}}_{\emph{M}}(Y ,d^{\prime},\psi)=\underline{\emph{mdim}}_{\emph{M}}(X ,d^{\prime},\psi)$, then $$ \overline{\emph{mdim}}_{\emph{M}}(X\times Y ,d\times d^{\prime}, \phi\times \psi) = \overline{\emph{mdim}}_{\emph{M}}(X ,d,\phi)+\overline{\emph{mdim}}_{\emph{M}}(Y  ,{d^{\prime}},\psi)$$ and $$ \underline{\emph{mdim}}_{\emph{M}}(X\times Y ,d\times d^{\prime}, \phi\times \psi) = \underline{\emph{mdim}}_{\emph{M}}(X ,d,\phi)+\underline{\emph{mdim}}_{\emph{M}}(Y  ,{d^{\prime}},\psi).$$ 
\end{enumerate}
\end{theorem} 
\begin{proof} For any $\varepsilon >0$, we always have  
  \begin{equation*}\label{equation1} \text{span}(\phi\times \psi,2\varepsilon) \leq \text{span}(\phi,\varepsilon) + \text{span}(  \psi,\varepsilon) \quad\text{and}\quad \text{sep}(\phi\times \psi,2\varepsilon) \geq \text{sep}(\phi,\varepsilon) + \text{sep}(  \psi,\varepsilon) .\end{equation*}
Hence, item i.  follows from  $$\limsup_{\varepsilon\to0} \frac{\text{span}(\phi\times \psi,2\varepsilon)}{|\log 2\varepsilon|}\leq \limsup_{\varepsilon\to0} \frac{\text{span}(\phi ,\varepsilon)}{|\log \varepsilon|}+\limsup_{\varepsilon\to0} \frac{\text{span}(  \psi,\varepsilon)}{|\log \varepsilon|},$$
item ii. follows from $$\limsup_{\varepsilon\to0} \frac{\text{sep}(\phi\times \psi,2\varepsilon)}{|\log 2\varepsilon|}\\
\geq \limsup_{\varepsilon\to0} \frac{\text{sep}(\phi ,\varepsilon)}{|\log \varepsilon|}+\liminf_{\varepsilon\to0} \frac{\text{sep}(  \psi,\varepsilon)}{|\log \varepsilon|},$$
item iii. follows from $$\liminf_{\varepsilon\to0} \frac{\text{sep}(\phi\times \psi,2\varepsilon)}{|\log 2\varepsilon|}\geq \liminf_{\varepsilon\to0} \frac{\text{sep}(\phi ,\varepsilon)}{|\log \varepsilon|}+\liminf_{\varepsilon\to0} \frac{\text{sep}(  \psi,\varepsilon)}{|\log \varepsilon|}$$ and item iv. follows from
$$\liminf_{\varepsilon\to0} \frac{\text{span}(\phi\times \psi,2\varepsilon)}{|\log 2\varepsilon|}\leq \liminf_{\varepsilon\to0} \frac{\text{span}(\phi ,\varepsilon)}{|\log \varepsilon|}+\limsup_{\varepsilon\to0} \frac{\text{span}(  \psi,\varepsilon)}{|\log \varepsilon|}.$$
Note that item v. is a consequence of items i-iv.\end{proof}

For the box dimension we have the following inequalities (see \cite{Falconer} and  \cite{WEI}):  
  \begin{equation}\label{dwfwfgh} \overline{\text{dim}}_\text{B}(X,d )+\underline{\text{dim}}_\text{{B}}(Y,d^{\prime} )\leq\overline{\text{dim}}_\text{{B}}(X\times Y,d\times d^{\prime} ) \leq \overline{\text{dim}}_\text{{B}}(X,d )+\overline{\text{dim}}_\text{{B}}(Y,d^{\prime} )\end{equation} and \begin{equation}\label{dwfwfgh1} 
 \underline{\text{dim}}_\text{{B}}(X,d )+\underline{\text{dim}}_\text{{B}}(Y,d^{\prime} )\leq   \underline{\text{dim}}_\text{{B}}(X\times Y ,d\times d^{\prime})\leq  \underline{\text{dim}}_\text{{B}}(X,d)+\overline{\text{dim}}_\text{{B}}(Y,d^{\prime} ).\end{equation} 
  If $ \overline{\text{dim}}_\text{{B}}(X.d)=\underline{\text{dim}}_\text{{B}}(X,d )$  or $ \overline{\text{dim}}_\text{{B}}(Y,d^{\prime} )=\underline{\text{dim}}_\text{{B}}(X,d)$, we can prove that  $$ \overline{\text{dim}}_\text{{B}}(X\times Y, d\times d^{\prime}) = \overline{\text{dim}}_\text{{B}}(X, d)+\overline{\text{dim}}_\text{{B}}(Y , {d^{\prime}})$$ and $$ \underline{\text{dim}}_\text{{B}}(X\times Y, d\times d^{\prime}) = \underline{\text{dim}}_\text{{B}}(X, d)+\underline{\text{dim}}_\text{{B}}(Y , {d^{\prime}}).$$
   Each inequality  in \eqref{dwfwfgh} and \eqref{dwfwfgh1}     can be strict (see \cite{WEI}). In the next example we will prove   the   inequalities i-iv in Theorem \ref{inequalitiess}    can be strict.

 \begin{example}\label{mnbad}  
 Let $(X,d)$ and $(Y,d^{\prime})$ be any compact metric spaces. The metric $\tilde{d}\times \tilde{d}^{\prime}$ on $X^{\mathbb{Z}}\times Y^{\mathbb{Z}}$ is defined by $$(\tilde{d}\times \tilde{d}^{\prime})((\bar{x},\bar{y}),(\bar{z},\bar{w}))=  \sum_{i\in\mathbb{Z}}\frac{1}{2^{|i|}}d(x_{i},z_{i})+\sum_{i\in\mathbb{Z}}\frac{1}{2^{|i|}}d^{\prime}(y_{i},w_{i}),$$  for $\bar{x}=(x_{i})_{i\in\mathbb{Z}},\,  \bar{z}=(z_{i})_{i\in\mathbb{Z}} \in X^{\mathbb{Z}}, \bar{y}=(y_{i})_{i\in\mathbb{Z}},\,  \bar{w}=(w_{i})_{i\in\mathbb{Z}} \in Y^{\mathbb{Z}}
$.  Furthermore, the metric  $(d\times d^{\prime})^{\ast}$ on  $(X\times Y)^{\mathbb{Z}}$ is given by \begin{align*}(d\times d^{\prime})^{\ast}((\overline{x,y}),(\overline{z,w}))= \sum_{i\in\mathbb{Z}}\frac{1}{2^{|i|}}d(x_{i},z_{i})+\sum_{i\in\mathbb{Z}}\frac{1}{2^{|i|}}d^{\prime}(y_{i},w_{i}),\end{align*}  for $(\overline{x,y})=(x_{i},y_{i})_{i\in\mathbb{Z}} $ and $(\overline{z,w})=(z_{i},w_{i})_{i\in\mathbb{Z}} $ in $(X\times Y)^{\mathbb{Z}}$. Consequently, the bijection 
$$ \Theta :  (X\times Y)^{\mathbb{Z}}\rightarrow  X^{\mathbb{Z}}\times Y^{\mathbb{Z}}, \quad \text{given by }(x_{i},y_{i})_{i\in\mathbb{Z}} \mapsto ((x_{i})_{i\in\mathbb{Z}},(y_{i})_{i\in\mathbb{Z}}), $$ 
is an isometry and furthermore  the diagram  \[ \begin{CD}
     (X\times Y)^{\mathbb{Z}} @> \sigma >>  (X\times Y)^{\mathbb{Z}} \\
    @VV \Theta V      @VV \Theta V      \\
    X^{\mathbb{Z}}\times Y^{\mathbb{Z}} @> \sigma_{1}\times \sigma_{2} >>  X^{\mathbb{Z}}\times Y^{\mathbb{Z}}
  \end{CD}
\]
is commutative, where $\sigma$ is the left shift on $(X\times Y)^{\mathbb{Z}} $, $\sigma_{1}$ is the left shift on $X^{\mathbb{Z}}$ and  $\sigma_{2}$ is the left shift on $Y^{\mathbb{Z}}$. It is clear that the metric mean dimension is invariant under isometric topological conjugacy. Therefore, 
\begin{align*}  \overline{\text{mdim}}_\text{{M}}(X^{\mathbb{Z}}\times Y^{\mathbb{Z}} , \tilde{d}_{1}\times \tilde{d}_{2}, \sigma_{1}\times \sigma_{2})& =  \overline{\text{mdim}}_{\text{M}}((X\times Y)^{\mathbb{Z}} ,(d\times d^{\prime})^{\ast}, \sigma)\\
&= \overline{\text{dim}}_{\text{B}} (X\times Y,d\times d^{\prime}) \end{align*}
    and 
 \begin{align*}  \underline{\text{mdim}}_{\text{M}}(X^{\mathbb{Z}}\times Y^{\mathbb{Z}} , \tilde{d}_{1}\times \tilde{d}_{2}, \sigma_{1}\times \sigma_{2}) = \underline{\text{dim}}_{\text{B}} (X\times Y,d\times d^{\prime}). \end{align*}   
If $(X,d)$ and $(Y,d^{\prime})$ are compact metric spaces such that each inequality  in \eqref{dwfwfgh} and \eqref{dwfwfgh1}  is  strict (see \cite{WEI}), then we can prove that the inequalities i-iv in Theorem  \ref{inequalitiess} are strict. For instances,   if   $$\overline{\text{dim}}_{\text{B}}(X\times Y,d\times d^{\prime})<\overline{\text{dim}}_{\text{B}}(X,d)+ \overline{\text{dim}}_{\text{B}}(Y,d^{\prime}),$$ then   \begin{align*}  \overline{\text{mdim}}_{\text{M}}(X^{\mathbb{Z}}\times Y^{\mathbb{Z}} , \tilde{d}_{1}\times \tilde{d}_{2}, \sigma_{1}\times \sigma_{2})& =  \overline{\text{dim}}_{\text{B}} (X\times Y,d\times d^{\prime}) <\overline{\text{dim}}_{\text{B}}(X,d)+ \overline{\text{dim}_{\text{B}}}(Y,d^{\prime})\\&= \overline{\text{mdim}}_{\text{M}}(X^{\mathbb{Z}} ,\tilde{d},\sigma_{1}) + \overline{\text{mdim}}_{\text{M}}(Y^{\mathbb{Z}} , \tilde{d}^{\prime}, \sigma_{2}).\end{align*}
\end{example}

\section{Density of continuous maps on manifolds with positive metric mean dimension}\label{section6} 
     Yano in \cite{Yano} proved that the set consisting of homeomorphisms with infinite topological entropy defined  on any  manifold with dimension biggest to one is residual in the set consisting of homeomorphisms on the manifold.  Furthermore, the set consisting of continuous maps defined  on the interval or  the circle with infinite topological entropy   is residual. In this section we will prove if $N$ is any riemannian manifold, then for any $a\in [0,\text{dim}(N)]$  the set consisting of continuous maps on $N$ whose metric mean dimension is equal to  $a$ is dense in $C^{0}(N)$.    Furthermore, the set consisting of continuous maps with upper metric mean dimension equal to $\text{dim}(N)$ is residual.
  
 \medskip
 
 On $C^{0}(X)$ we will consider the metric \begin{equation*}\label{cbenfn} \hat{d}(\phi,\varphi)=\max_{x\in X}d(\phi(x),\varphi(x))\quad \quad\text{ for any }\phi, \varphi \in   C^{0}(X).\end{equation*}   
 
For any $a\in [0,\overline{\text{dim}}_{\text{B}}(X,d)]$, set $$C_{a}(X)=\{\phi\in C^{0}(X): \overline{\text{mdim}}_\text{M}(X ,d,\phi)=  \underline{\text{mdim}}_\text{M}(X ,d,\phi)=a\}.$$

Note for any riemannian manifold $N$ with riemannian metric $d$,   we have   $C_{0}(N)$ is   dense  in $C^{0}(N)$, since the set consisting of $C^{1}$-maps on $N$ is dense in $C^{0}(N)$ and the metric mean dimension of any $C^{1}$-map is equal to zero. 
 
 \medskip

Example \ref{EXAMPLE1} proves  for any $a\in [0,1]$ there exists  a $\phi_{a}\in C^{0}([0,1])$ such that $$\overline{\text{mdim}}_\text{{M}}([0,1] ,|\cdot |,\phi_{a})=\underline{\text{mdim}}_{\text{M}}([0,1] ,|\cdot |,\phi_{a})=a.$$ 

In   \cite{Carvalho}, Theorem C, the authors proved for  each $ a\in [0, 1]$ the set consisting of  continuous maps on $[0,1]$   with lower and upper   metric mean dimension equal to  $ a$ is dense in $C^{0}([0,1])$.   We will present a   proof of this   fact for the sake of completeness.   

 \begin{theorem}\label{densidadd}
    $ C_{a}([0,1])$ is   dense  in $C^{0}([0,1])$  for each  $a\in [0,1]$.
 \end{theorem} 
 \begin{proof}  
  We had seen that $C_{0}([0,1])$ is dense in $C^{0}([0,1])$.    Therefore, in order to prove the theorem, it is sufficient to show   if  $\phi_{0} \in C_{0}([0,1])$, then for any $\varepsilon>0$  there exists $\psi_{a}\in C_{a}([0,1])$ such that $d(\phi_{0},\psi_{a})<\varepsilon$. Fix $ \phi_{0}:[0,1]\rightarrow [0,1]\in C_{0}([0,1])$ and take  $\varepsilon >0$.  \begin{figure}[ht]
 \centering
 { 
  \includegraphics[width=0.265\textwidth]{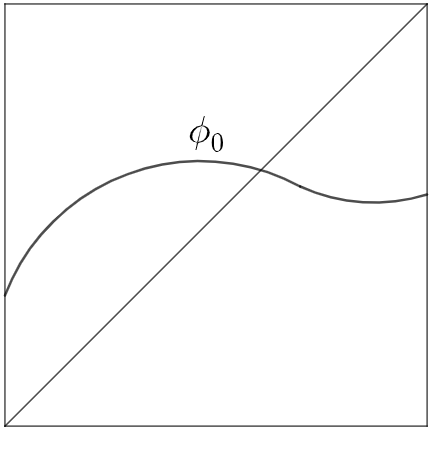}}  
 \quad{\label{fig:ejemplod12}
 \includegraphics[width=0.27\textwidth]{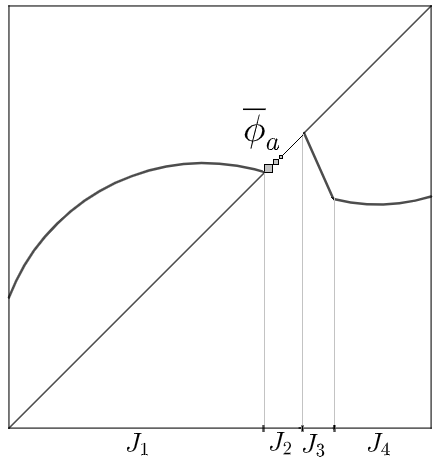}} 
\caption{Graphs of $\phi_{0}$ and $\psi_{a}$. $\overline{\phi}_{a}= T_{2}^{-1}\phi_{a} T_{2}$} \label{fig:ejemplo1wd}
\end{figure}
 Let $p^{\ast}$ be a fixed point of $\phi_{0}$.   Choose  $\delta>0$ such that $|\phi_{0}(x)-\phi_{0} (p^{\ast})|<\varepsilon/2$ for any $x$ with $|x-p^{\ast}|<\delta$.  Take $\phi_{a}\in C_{a}([0,1])$ for some  $a\in (0,1]$ (it is follows from    Examples \ref{EXAMPLE1} and \ref{med1}   that for any  $a\in (0,1]$ there exists  $\phi_{a}\in C_{a}([0,1])$). Set  $J_{1}=[0,p^{\ast}]$, $J_{2}=[p^{\ast},p^{\ast}+\delta/2]$, $J_{3}=[p^{\ast}+\delta/2,p^{\ast}+\delta] $ and $J_{4}=[p^{\ast}+\delta,1]$. Take the continuous map $\psi_{a}$ on $X$ defined as $$
\psi_{a}(x)= \begin{cases}
    \phi_{0}(x), &  \text{ if }x\in J_{1}\cup J_{4}, \\
   T_{2}^{-1}\phi_{a}T_{2}(x), &  \text{ if }x\in J_{2}, \\
    T_{3}(x), &  \text{ if }x\in J_{3}, 
      \end{cases}
$$ where    $T_{2}:J_{2}\rightarrow I $  is the  affine map such that $T_{2}(p^{\ast})=0 $ and $T_{2}(p^{\ast}+\delta/2)=1 $,   and  $ T_{3}:J_{3}\rightarrow [ p^{\ast}+\delta/2, \phi_{0}(p^{\ast}+\delta)]$ is the  affine map   such that $ T_{3}(p^{\ast}+\delta/2)=p^{\ast}+\delta/2$  and $ T_{3} (p^{\ast}+\delta)=\phi_{0}(p^{\ast}+\delta)$ (see Figure \ref{fig:ejemplo1wd}).  Note that $\hat{d}(\psi_{a},\phi_{0})<\varepsilon.$ Set $A=\cup_{i=0}^{\infty}\psi_{a}^{-i}(J_{2})$ and $B=A^{c}$. Note that $$\Omega(\psi_{a}|_{A}) =\Omega(\psi_{a}|_{J_{2}})\subseteq J_{2}.$$   Hence  \begin{align*}\text{mdim}_\text{M}([0,1] ,|\cdot |,\psi_{a})&=\max \{\text{mdim}_\text{M}(A ,|\cdot |,\psi_{a}|_{A})  ,\text{mdim}_\text{M}(B ,|\cdot |,\psi_{a}|_{B})\}\\
&= \text{mdim}_\text{M}(J_{2} ,|\cdot |,\psi_{a}) 
=a.\end{align*} 
  This fact proves the theorem.
 \end{proof}

\begin{remark}
Note that in Theorem \ref{densidadd} we prove the set $\mathcal{A}$  consisting of maps $\psi\in C^{0}([0,1])$ such that, for some $a,b\in [0,1]$, $\psi|_{[a,b]}:[a,b]\rightarrow [a,b]$  satisfies the conditions in Theorem \ref{misiu} and outside of $[a,b]$ $\psi $  has zero entropy, is dense in $C^{0}([0,1])$. 
Therefore, Corollary \ref{corollad} can be applied for any map in $\mathcal{A}$, which is  a dense subset of  $C^{0}([0,1])$.
\end{remark}

  For $a,b\in (0,1]$, let $\phi_{a},\phi_{b}\in C^{0}([0,1])$ be such that $$  {\text{mdim}}_{\text{M}}([0,1] ,d,\phi_{a})=a \quad \text{and}\quad   {\text{mdim}}_{\text{M}}([0,1] ,d,\phi_{b})=b$$ (see Example \ref{EXAMPLE1}). 
It follows from Theorem \ref{inequalitiess}, item v, that  $$  \text{mdim}_{\text{M}}([0,1]\times [0,1] ,d\times d,\phi_{a}\times \phi_{b})  =   \text{mdim}_{\text{M}}([0,1]  ,d ,\phi_{a} )+ \text{mdim}_{\text{M}}([0,1],d, \phi_{b})=a+b.$$
Hence, we have: 

\begin{lemma}\label{densitypositivemanifold1}
 Fix $n\in\mathbb{N}$. For any $a\in [0,n]$, there exists $\phi_{a}\in C^{0}([0,1]^{n})$ such that $$\overline{\emph{mdim}}_\emph{M}([0,1]^{n}  ,d^{n},\phi_{a})=\underline{\emph{mdim}}_\emph{M}([0,1]^{n}  ,d^{n},\phi_{a})=a.$$ Furthermore, given that $d^{n}$ (see \eqref{bnm}) and $\Vert \cdot \Vert$, where $\Vert (x_{1},\dots, x_{n})\Vert = \sqrt{x_{1}^{2}+\cdots + x_{n}^{2}}$ for any $(x_{1},\dots, x_{n})\in\mathbb{R}^{n}$,  are uniformly equivalent, we have  for any $a\in [0,n]$, there exists $\phi_{a}\in C^{0}(X^{n})$ such that $$\overline{\emph{mdim}}_\emph{M}([0,1]^{n}  ,\Vert \cdot \Vert ,\phi_{a})=\underline{\emph{mdim}}_\emph{M}([0,1]^{n}  ,\Vert \cdot \Vert ,\phi_{a})=a.$$
\end{lemma}
 
 \begin{remark}
    Fix  $r_{1},r_{2},\dots, r_{n}\in(0,\infty)$ and  $s_{1},s_{2}, \dots, s_{n}\in \mathbb{N}$ and for $i=1,2,\dots, n$ take   $n$ maps $\phi_{s_{i},r_{i}} \in C^{0}([0,1])$    defined as in Example \ref{EXAMPLE1}. Thus  for each $i=1,2,\dots, n$ we have $$ {{\text{mdim}}_{\text{M}}}([0,1] ,|\cdot |,\phi_{s_{i},r_{i}}) =\frac{s_{i}}{r_{i}+s_{i}} . $$
    From Theorem \ref{inequalitiess}, item v, we have $$ {{\text{mdim}}_{\text{M}}}([0,1]^{n} ,d^{n},\phi_{s_{1},r_{1}}\times \phi_{s_{2},r_{2}}\times\cdots \times\phi_{s_{n},r_{n}}) =\sum_{i=1}^{n}\frac{s_{i}}{r_{i}+s_{i}}.  $$
    Furthermore, it follows from \ref{sdfwddds} and Theorem \ref{inequalitiess} that for any $k\in\mathbb{N}$ we have   
    \begin{align*} {{\text{mdim}}_{\text{M}}}([0,1]^{n} ,d^{n},(\phi_{s_{1},r_{1}} \times\cdots \times\phi_{s_{n},r_{n}})^{k}) 
    &={{\text{mdim}}_{\text{M}}}([0,1]^{n} ,d^{n},(\phi_{s_{1},r_{1}})^{k}\times\dots \times (\phi_{s_{n},r_{n}})^{k}) \\
    &={{\text{mdim}}_{\text{M}}}([0,1]^{n} ,d^{n},\phi_{ks_{1},r_{1}}\times\dots\times \phi_{ks_{n},r_{n}}) \\
    &=\sum_{i=1}^{n}\frac{ks_{i}}{r_{i}+ks_{i}}.  \end{align*}
 \end{remark}

Throughout this section, we will fix a compact riemannian manifold $N$ with riemannian metric $d$ and $\text{dim}(N)=n\geq 1$. The proof of the following theorem consists in perturbing a map  on small neighborhoods (on which we will work by means of coordinate charts) of the orbit of a periodic point, that is, on finitely many neighborhoods. Since the metric mean dimension depends on the metric, we must be careful to choose the charts that will be used to make the perturbations. For this reason we will take the charts given by the exponential map, which provides us the required properties. 
 Indeed, for each $p\in N$, consider the      \textit{exponential map}  $$\text{exp}_{p}: B_{\delta^{\prime}}(0_{p})\subseteq T_{p}N\rightarrow B_{\delta^{\prime}}(p)\subseteq N,$$ where $0_{p}$ is the origin in the tangent space $T_{p}N$,  $\delta^{\prime}$ is the \textit{injectivity radius} of $N$ and      $B_{\epsilon}(x)$ denote the open ball of radius $\epsilon>0$ with center $x$. 
 We will take $ {\delta}_{N}=\frac{\delta^{\prime}}{2}$.    
The exponential map has the following properties (see \cite{Manfredo}, Chapter III):
\begin{itemize}
\item Since $N$ is compact, $\delta^{\prime}$ does not depends on $p$.
\item $ \text{exp}_{p}(0_{p}) = p$ and $ \text{exp}_{p}[ B_{\delta_{N}}(0_{p})] = B_{\delta_{N}}({p})$;
    \item $\text{exp}_{p}: B_{\delta_{N}}(0_{p})\rightarrow B_{\delta_{N}}({p})$ is a diffeomorphism;
    \item If $v\in B_{\delta_{N}}(0_{p})$, taking $q=\text{exp}_{p}(v)$ we have $d(p,q)=\Vert v\Vert$. 
    \item The derivative of $\text{exp}_{p}$ at the origin is the identity map: $$D (\text{exp}_{p})(0) = \text{id} : T_{p}N \rightarrow  T_{p}N.$$
\end{itemize}

Since $\text{exp}_{p}: B_{\delta_{N}}(0_{p})\rightarrow B_{\delta_{N}}({p})$ is a diffeomorphism and $D (\text{exp}_{p})(0) = \text{id} : T_{p}N \rightarrow  T_{p}N,$ we have $\text{exp}_{p}: B_{\delta_{N}}(0_{p})\rightarrow B_{\delta_{N}}({p})$ is a bi-Lipschitz map with Lipschitz constant close to 1. Therefore, we can assume that if $v_{1},v_{2}\in B_{\delta_{N}}(0_{p})$, taking $q_{1}=\text{exp}_{p}(v_{1})$ and $q_{2}=\text{exp}_{p}(v_{2})$,  we have $d(q_{1},q_{2})=\Vert v_{1} -v_{2}\Vert$. Furthermore, we  will  identify  $B_{\delta_{N}}(0_{p})\subset T_{p}N$ with $B_{\delta_{N}}(0)=\{x\in\mathbb{R}^{n}:\Vert x\Vert <\delta_{N}\}\subseteq \mathbb{R}^{n}$.

\begin{theorem}
\label{densitypositivemanifold}  For any $a\in [0,n]$, the set $$C_{a}(N)=\{\phi\in C^{0}(N):  \overline{\emph{mdim}}_{\emph{M}}(N ,d,\phi)=\underline{\emph{mdim}}_{\emph{M}}(N ,d,\phi)=a\}$$ is   dense  in $C^{0}(N)$.
\end{theorem}
\begin{proof} 
Let $P^{r}(N)$ be the set  consisting of $C^{r}$-differentiable maps on $N$ with a periodic point. This set is $C^{0}$-dense in $C^{0}(N)$ (see \cite{Artin},    \cite{Hurley}).  Hence, in order to prove the theorem it is sufficient to show if $\phi_{0}\in P^{r}(N)$, then for any $\varepsilon >0$   there exists $\varphi_{a}\in C_{a}(N)$, with $d(\phi_{0},\varphi_{a})<\varepsilon$.   Fix $\phi_{0}\in P_{r}(N)$
 and take  $\varepsilon\in (0,\delta_{N})$. Suppose that $p$ is a periodic point of $\phi_{0}$ with period $k$. We can suppose that $B_{\varepsilon}(\phi_{0}^{i}(p))\cap B_{\varepsilon}(\phi_{0}^{j}(p))=\emptyset$, for $i,j=1,\dots,k$ with $i\neq j$.  Take $\lambda\in (0,\varepsilon/4)$  such that   $ \phi_{0} (B_{4\lambda}(\phi_{0}^{i}(p)))\subseteq B_{\varepsilon/4}(\phi_{0}^{i+1}(p))$   for $i=0,\dots,k-1$.    Take $\phi_{a}\in C^{0}([-\frac{\lambda}{3},\frac{\lambda}{3}]^{n})$ obtained  by a cartesian product of maps  given  in Example   \ref{EXAMPLE1}, with $[-\frac{\lambda}{3},\frac{\lambda}{3}]$ instead of $[0,1]$, such that  $\text{mdim}_\text{M}([-\frac{\lambda}{3},\frac{\lambda}{3}]^{n}  ,\Vert \cdot\Vert,\phi_{a})=a$ (see Lemma \ref{densitypositivemanifold1}). Set  $$A=\left[-\frac{\lambda}{3},\frac{\lambda}{3}\right]^{n},\quad\text{  }\quad B=[-\lambda,\lambda]^{n}\setminus (-2\lambda/3,2{\lambda}/3)^{n},\quad\quad C=[-{\lambda},{\lambda}]^{n}\setminus  (A\cup B) .$$ 
 Take the map  $ \varphi_{a}:A\cup B\rightarrow A\cup B$ defined by  $$ 
 \varphi_{a}(x)= \begin{cases}
     {\phi}_{a}(x), &  \text{ if }x\in  A \\
    x, &  \text{ if }  x\in B .  
      \end{cases}$$ 
Note that $$\varphi_{a}(\partial A)=\partial A \quad \text{and}\quad \varphi_{a}\left(\partial   \left([-2\lambda/3,2{\lambda}/3]^{n}\right)\right)=\partial   \left([-2\lambda/3,2\lambda/3]^{n}\right).$$ 
Furthermore, if $(x_{1},\dots, x_{n})\in \partial A$, then $x_{i}\in \{-\lambda/3, \lambda/3\}$ for some $i$ and we have \begin{equation}\label{ncbsiff} \varphi_{a}(x_{1},\dots, x_{i},\dots,  x_{n})= (z_{1},\dots,z_{i-1}, x_{i},z_{i+1},\dots,  z_{n}) , \end{equation} for some $z_{j}\in \left[-\frac{\lambda}{3},\frac{\lambda}{3}\right]$, for $j\in   \{1,\dots, i-1,i+1,\dots,n\}$. Hence $(x_{1},\dots, x_{i},\dots,  x_{n})$ and  $ \varphi_{a}(x_{1},\dots, x_{i},\dots,  x_{n})$ belong to the same face  of   $ \partial A$. Considering this fact, we extend  $\varphi_{a}$ to a continuous map $\bar{\varphi}_{a}:[-\lambda,\lambda]^{n}\rightarrow [-\lambda,\lambda]^{n} $.    \begin{figure}[ht]    \centering \includegraphics[width=0.5\textwidth]{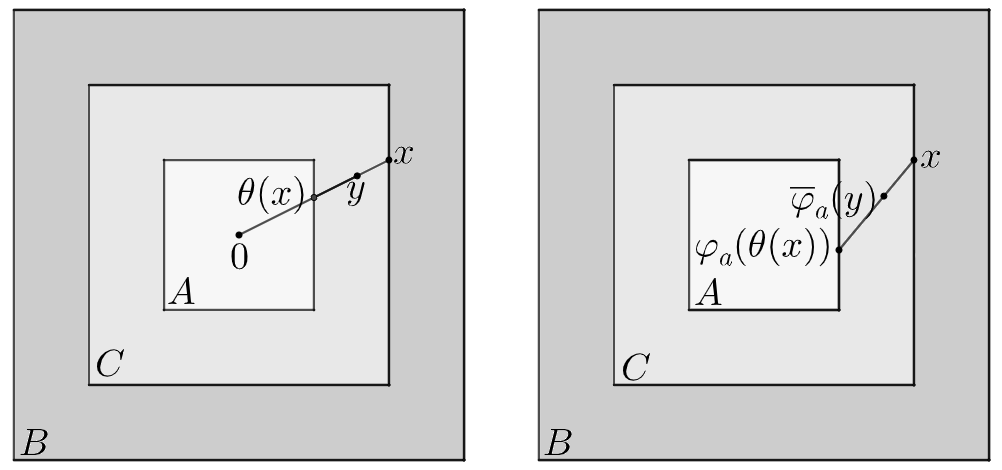}\caption{Extension of $\overline{\varphi}_{a}$}   \label{ABN}\end{figure} 
For any $x\in \partial     ([-2\lambda/3,2\lambda/3]^{n})$, take the line segment passing through $x$ and $0\in\mathbb{R}^{n}$. This line passes through a unique point $\theta(x)\in \partial    A$. Any $y\in C $  can be written as $y=t\theta(x)+(1-t)x$, for some $t\in[0,1]$, where $x\in \partial     ([-2\lambda/3,2\lambda/3]^{n})$ (see Figure \ref{ABN}, noting that from \eqref{ncbsiff} we have   $\theta(x)$ and $\varphi_{a}(\theta(x))$ belong to the same face of $\partial A$). We define $$\bar{\varphi}_{a}(y)= \bar{\varphi}_{a}(t\theta(x)+(1-t)x)=t\varphi_{a}(\theta(x))+(1-t) x, \quad\quad\text{for }y\in C.$$ We have $\bar{\varphi}_{a}:[-\lambda,\lambda]^{n}\rightarrow [-\lambda,\lambda]^{n} $ is a continuous map:    
For $y,z\in C$, there exist $x_{y},x_{z}\in   \partial     ([-2\lambda/3,2\lambda/3]^{n})$ and  $t,s\in [0,1]$ such that $$ y=t\theta(x_{y})+(1-t)x_{y}\quad\text{and}\quad z=s\theta(x_{z})+(1-s)x_{z}.$$ 
If $y$ and $z$  are close, then  $x_{y},x_{z}$ are close and therefore    $\theta(x_{y})$ and $\theta(x_{z})$ are close (note that $\theta:  \partial     ([-2\lambda/3,2\lambda/3]^{n})\rightarrow \partial A$ is a continuous map), which implies that $t$ and $s$ are close. Given that $\varphi_{a}$ is a continuous map, we have $ \varphi_{a}(\theta(x_{y})) $ and $ \varphi_{a}(\theta(x_{z}))$  and therefore  $\bar{\varphi}_{a}(y)$ and $\bar{\varphi}_{a}(z)$ are close. Next, if $y\in \partial A$, then  $t=1$ and thus $y=\theta(x)$. Therefore $$\bar{\varphi}_{a}(y)=\bar{\varphi}_{a}(\theta(x))={\varphi}_{a}(\theta(x)).$$    If $y\in \partial ([-2\lambda/3,2\lambda/3]^{n})$, then  $t=0$. Thus $y=x$ and therefore  $$\bar{\varphi}_{a}(y)=\bar{\varphi}_{a}(x)=x={\varphi}_{a}(x).$$

From \ref{ncbsiff} we have if $t\in[0,1]$ then  $t\theta(x)+(1-t)x$ and $\varphi_{a}(t\theta(x)+(1-t)x)$ belong to $C$. Given that $\bar{\varphi}_{a}(\partial C)=\partial C$ ($\bar{\varphi}_{a}$ is the identity on $\partial B$ and it is equal to $\varphi_{a}$ on $\partial A$, which is surjective), we have by the continuity of $\bar{\varphi}_{a}$ that    $\bar{\varphi}_{a}( C)=  C$. Therefore,   $ \text{mdim}_\text{M}( C ,\Vert \cdot\Vert,\bar{\varphi}_{a})\leq a$.  Hence,   \begin{align*}\text{mdim}_\text{M}([-\sigma,\sigma]^{n} ,\Vert \cdot\Vert,\bar{\varphi}_{a})&=\max \{\text{mdim}_\text{M}(A ,\Vert \cdot\Vert,\bar{\varphi}_{a})  ,\text{mdim}_\text{M}(B\cup C ,\Vert \cdot\Vert,\bar{\varphi}_{a})\}\\
 &=\text{mdim}_\text{M}(A  ,\Vert \cdot\Vert,\bar{\varphi}_{a}) =\text{mdim}_\text{M}(A ,\Vert \cdot\Vert, {\phi}_{a})   =a .\end{align*} 
Consider $$
\psi_{a}(q)= \begin{cases}
    \text{exp}_{\phi_{0}^{i+1}(p)}\circ\bar{\varphi}_{a}\circ\text{exp}^{-1}_{\phi_{0}^{i}(p)}(q), &  \text{if }q\in  N_{i}=\text{exp}_{\phi_{0}^{i}(p)}([-\lambda,\lambda]^{n}), \text{ for some }i  \\
    \phi_{0}(q), &  \text{if }  q\in B= N\setminus \left( \underset{i=1,\dots,k}{\bigcup}\text{exp}_{\phi_{0}^{i}(p)}((-2\lambda,2\lambda )^{n})\right) . 
      \end{cases}$$
 Next, we extend  $\psi_{a}$ to a continuous map $\bar{\psi}_{a}:N\rightarrow N $. For any $u\in \partial ([-2\lambda,2\lambda ]^{n})$, take the line segment passing through $u$ and $0\in\mathbb{R}^{n}$.  This line passes through a unique point $\beta(u)\in \partial  [-\lambda,\lambda]^{n} $. Any $w\in C_{i}=\text{exp}_{\phi_{0}^{i}(p)}\left[ [-2\lambda,2\lambda]^{n}\setminus   [-\lambda,\lambda]^{n} \right]$  can be written as $$w=\text{exp}_{\phi_{0}^{i}(p)}(t\beta(u)+(1-t)u),$$ for some $t\in[0,1]$, where $u\in \partial ([-2\lambda,2\lambda]^{n})$. For $w\in C_{i}$ we set \begin{align*}\bar{\psi}_{a}(w)&= \bar{\psi}_{a}(\text{exp}_{\phi_{0}^{i}(p)}(t\beta(u)+(1-t)u))\\
 &=\text{exp}_{\phi_{0}^{i+1}(p)}[t\bar{\varphi}_{a}(\beta(u))  +(1-t)\text{exp}^{-1}_{\phi_{0}^{i+1}(p)}(\phi_{0}(\text{exp}_{\phi_{0}^{i}(p)}(u)))]\\
 &=\text{exp}_{\phi_{0}^{i+1}(p)}[t \beta(u)  +(1-t)\text{exp}^{-1}_{\phi_{0}^{i+1}(p)}(\phi_{0}(\text{exp}_{\phi_{0}^{i}(p)}(u)))] .   \end{align*}  We have  $\bar{\psi}_{a}:N\rightarrow N $ is a continuous map (note that $\beta:\partial ([-2\lambda,2\lambda ]^{n})\rightarrow \partial  [-\lambda,\lambda]^{n} $ is a continuous map).   
Furthermore, we have  $\hat{d}(\phi_{0},\bar{\psi}_{a})<\varepsilon$.  Note if $q\in N_{i}$, we have
$$  (\overline{\psi}_{a})^{s}(q) =   \text{exp}_{\phi_{0}^{(i+s)\text{ mod }k}(p)}\circ(\bar{\varphi}_{a})^{s}\circ\text{exp}^{-1}_{\phi_{0}^{i}(p)}(q) \quad\text{and}\quad   (\overline{\psi}_{a})^{k}(q) =   \text{exp}_{\phi_{0}^{i}(p)}\circ(\bar{\varphi}_{a})^{k}\circ\text{exp}^{-1}_{\phi_{0}^{i}(p)}(q) .$$
Hence,  $A\subseteq [-\lambda,\lambda]^{n}$   is {an} $(s,\overline{\varphi}_{a},\epsilon)$-{separated} set if and only if $\text{exp}_{\phi_{0}^{i}(p)}(A)\subseteq N$   is {an} $(s,\overline{\psi}_{a},\epsilon)$-{separated} set for any $\epsilon>0$. Therefore, setting $L= \overset{k}{\underset{i=1}{\bigcup}}N_{i},$ we have $$\text{sep}(s,\overline{\psi}_{a}|_{L},\epsilon)=k\,\text{sep}(s,\overline{\varphi}_{a}  ,\epsilon) \quad\text{and thus}\quad  \text{mdim}_\text{M}( L ,d,\bar{\psi}_{a}|_{L})= \text{mdim}_\text{M}([-\lambda,\lambda]^{n} ,\Vert \cdot\Vert,\bar{\varphi}_{a}).$$    Set $K=\cup_{i=0}^{\infty}\bar{\psi}_{a}^{-i}\left(L\right)$ and $Z=K^{c}$. Note that $ \Omega(\bar{\psi}_{a}|_{K}) \subseteq   L  $     and   $ \phi_{0}|_{Z}$ is a differentiable map. Hence  $\text{mdim}_\text{M}(Z, d, \bar{\psi}_{a}|_{Z}
) = 0$  and therefore   \begin{align*}\text{mdim}_\text{M}(N ,d,\bar{\psi}_{a} )&=\max \{\text{mdim}_\text{M}(K ,d,\bar{\psi}_{a}|_{K})  ,\text{mdim}_\text{M}(Z ,d,\bar{\psi}_{a}|_{Z})\}\\
&=\max \{\text{mdim}_\text{M}(L ,d,\bar{\psi}_{a}|_{L} )  ,\text{mdim}_\text{M}(Z ,d,\bar{\psi}_{a}|_{Z} )\}\\
&= \text{mdim}_\text{M}(L ,d,\bar{\psi}_{a}|_{L} ) =a,\end{align*} 
 which proves the theorem. 
\end{proof}
 
In  \cite{Carvalho},  Theorem A, the authors proved  if  $\text{dim}(N)\geq 2$, then the set consisting  of  homeomorphisms with upper metric mean dimension equal to $n= \text{dim}(N)$ is residual in $ \text{Hom}(N)$.   Furthermore, they showed  the set consisting of continuous maps on $[0,1]$ with upper metric mean dimension equal to 1 is  residual in $C^{0}([0,1])$. Inspired by the proof of these facts,
we will show    the set consisting of continuous maps on $N$ with   upper metric mean dimension equal to $n$, which we will denote by $\overline{C}_{n}(N)$, is residual in $C^{0}(N)$.  
   
\medskip
 
A closed $n$-rectangular box  is a product  $J^{n}=J_{1}\times \cdots \times J_{n}$ of closed subintervals  $J_{i}$ for any $i=1,\dots, n$.   From now on,   we denote by   $J^{n}$  a closed $n$-rectangular box and we set $$|J^{n}|:=\min_{i=1,\dots,n}|J_{i}|,\quad\text{where}\quad J^{n}=J_{1}\times \cdots \times J_{n}.$$ For any closed interval $J=[a,b]$, let $\hat{J}=[\frac{2a+b}{3},\frac{a+2b}{3}]$, that is, the second third  of $J$. For a closed $n$-rectangular box ${J^{n}}=J_{1}\times \cdots \times J_{n}$, set $\hat{J^{n}}=\hat{J}_{1}\times \cdots \times \hat{J}_{n}$ (see Figure \ref{nhorseshoe}(a)).

 \begin{figure}[ht]
  \centering
  \begin{tabular}[b]{c}
    \includegraphics[width=.4\linewidth]{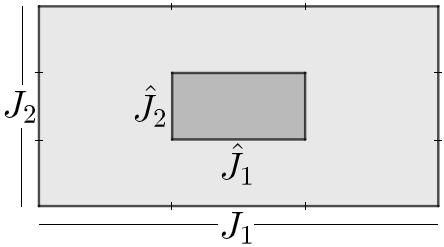} \\
    \small (a) $J^{2}=J_{1}\times J_{2}.$ $\hat{J}^{2}=\hat{J}_{1}\times \hat{J}_{2} $ \label{third}
  \end{tabular} \,
  \begin{tabular}[b]{c}
    \includegraphics[width=.45\linewidth]{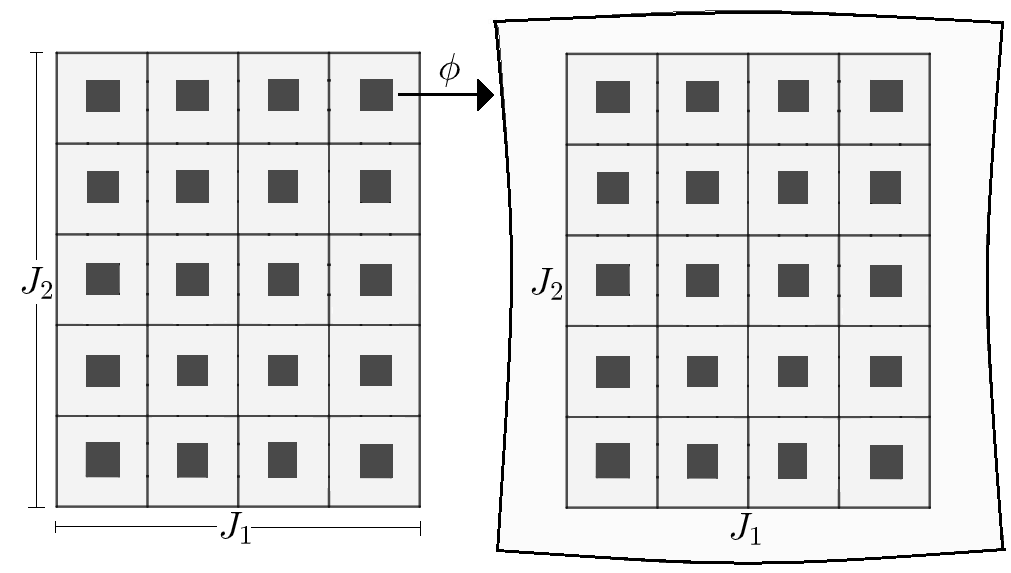} \\
    \small (b) A strong $(2,\epsilon, 20)$-horseshoe
   \end{tabular}
  \caption{Strong horseshoe}\label{nhorseshoe}
\end{figure}

For $\epsilon\in(0,1) $ and $k\in\mathbb{N}$, we say   a closed $n$-rectangular box $J^{n}\subset U\subset \mathbb{R}^{n}$ is a \textit{strong} $(n,\epsilon, k)$-\textit{horseshoe} 
of a continuous map $\phi:U\rightarrow \mathbb{R}^{n}$ if $|J^{n}|>\epsilon$ and $J^{n}$ contains  $k$ closed $n$-rectangular boxes $J_{1}^{n}, \dots, J_{k}^{n}\subseteq J^{n}$,  with $(J_{s}^{n})^{\circ}\cap (J_{r}^{n})^{\circ}=\emptyset$ for $s\neq r$, such that  $|J_{i}^{n}|>\frac{|J^{n}|}{2\sqrt[n]{k}} $ and  $J^{n}\subset  (\phi(\hat{J}_{i}^{n}))^{\circ}$ for any $i=1,\dots, k$. In  Figure \ref{nhorseshoe}(b) we present an example of a strong $(2,\epsilon,20)$-horseshoe.

  \medskip

 We say $\phi\in C^{0}(N)$ has a strong $(n,\epsilon, k)$-horseshoe $J^{n}$, where $J^{n}\subset \mathbb{R}^{n}$ is a closed $n$-rectangular box, if there exist $s$  exponential charts $\text{exp}_{i}: B(0,\delta_{N})\rightarrow N$, for $i=1,\dots,s$,   such that: \begin{itemize} \item $\phi_{i}=\text{exp}_{(i+1)\text{mod}\,  s}\circ \phi \circ \text{exp}^{-1}_{i}: B(0,\delta )\rightarrow B(0,\delta_{N})$ is well defined for some $\delta \leq \delta_{N}$;
 \item $J^{n}\subset (\phi_{i}(J^{n}))^{\circ}$ for each $i=1,\dots, s$;
 \item $J^{n}$ is a strong $(n,\epsilon, k)$-horseshoe for $\phi_{i}$ for each $i=1,\dots, s$. \end{itemize}  
 To simplify the notation, we will set $\phi_{i}=\phi$ for each $i=1,\dots, s$. 
 
 \medskip
 
 For $\epsilon >0$ and $k\in\mathbb{N}$, set 
\begin{align*} &H(n,\epsilon, k)=\{\phi\in C^{0}(N):  \phi \text{ has a strong }(n,\epsilon,k) \text{-horseshoe} \}  \\
&H(n,k)=\bigcup_{i\in\mathbb{N}}H\left(n, \frac{1}{ i^{2}},3^{n\,k\,i} \right)\\
 & \mathcal{H}^{n}=\overset{\infty }{\underset{k=1}{\bigcap}} H(n,k). \end{align*}

  \begin{theorem}\label{teoresidual} $\mathcal{H}^{n}$ is residual and, if $\phi\in \mathcal{H}^{n}$, then $\overline{\emph{mdim}}_{\emph{M}}(N,d,\phi)=n$. Therefore, for any $n\geq 1$, if $N$ is a $n$-dimensional compact riemannian manifold with riemannian metric $d$, the set  $\overline{C}_{n}(N)=\{\phi\in C^{0}(N):\overline{\emph{mdim}}_{\emph{M}}(N,d,\phi)=n\}$ is residual in $C^{0}(N)$. 
\end{theorem}
\begin{proof}  We prove for any $\epsilon\in (0,\delta_{N})$ and $k\in\mathbb{N}$, we have   $ H(n,\epsilon, k)$  is nonempty. In fact, consider the map $g:[0,1]\rightarrow [0,1]$ defined in Example \ref{EXAMPLE1}. For any $s\geq 2$, $g^{s}$ has a strong $\left(1, 1-\frac{4}{3^{s}},\frac{3^{s}-3}{3}\right)$-horseshoe (see Figure \ref{ABNN}):   \begin{align*}J&=\left[\frac{1}{3^{s}},\frac{4}{3^{s}}\right]\cup  \left[\frac{4}{3^{s}},\frac{7}{3^{s}}\right] \cup \cdots\cup \left[\frac{3^{s}-5}{3^{s}},\frac{3^{s}-2}{3^{s}}\right], \quad 
|J|=3^{s}-3> 1-\frac{4}{3^{s}}, \\
 \left|J_{r}\right|:&=\left|\left[\frac{3(r-1)+1}{3^{s}},\frac{3(r-1)+4}{3^{s}}\right]\right|=\frac{3}{3^{s}}>\frac{1}{2|J|}=\frac{1}{2(3^{s}-3)},\\
 J&\subset(0,1) = (g^{s}(\hat{J}_{r}))^{\circ}=\left(g^{s}\left(\left[\frac{3(r-1)+2}{3^{s}},\frac{ 3(r-1)+3}{3^{s}}\right] \right)\right)^{\circ},
 \end{align*}
 for any $r=1,\dots ,\frac{3^{s}-3}{3}$.  \begin{figure}[ht]    \centering \includegraphics[width=0.24\textwidth]{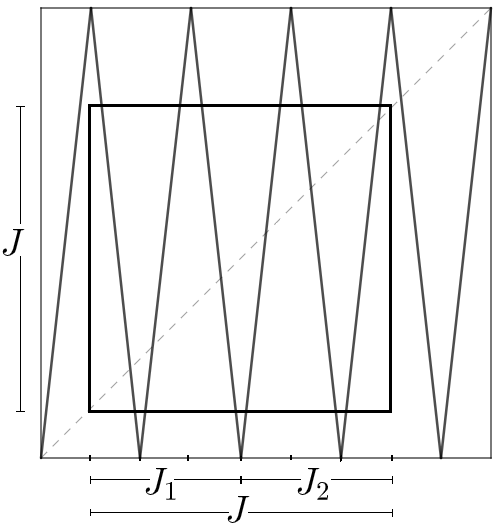}\caption{$J$ is a strong (1,1-4/9,2)-horseshoe for $g^{2}$}   \label{ABNN}\end{figure} Take an $s$ large enough such that $1-\frac{4}{3^{s}}>\epsilon$ and $3^{s}-3\geq k$. We have $J^{n}$ is a strong $(n,\epsilon, k)$-horseshoe of $\tilde{g}:=g\times\cdots \times g\in C^{0}([0,1]^{n})$. We can make a affine change of variable and we can assume  that $g:[-\delta_{N},\delta_{N}]\rightarrow [-\delta_{N},\delta_{N}]$. Let $\psi:B(0,\delta_{N})\rightarrow N$ be an exponential chart. The map $\psi\circ\tilde{g}\circ\psi^{-1}: \psi([-\delta_{N},\delta_{N}]^{n})\rightarrow \psi([-\delta_{N},\delta_{N}]^{n})$ can be extended to a continuous map $\hat{g}$ on $N$ as we made in Theorem \ref{densitypositivemanifold}  (note $\tilde{g}$ has the properties needed in order to do this extension). We have $\hat{g}\in H(n,\epsilon,k)$.

\medskip

  $ H(n,\epsilon, k)$  is    open in $C^{0}(N)$: if $\phi\in H(n,\epsilon, k)$ and $J^{n}$ is a strong $(n,\epsilon,k)$-horseshoe of $\phi$ we can take a small enough open neighborhood $U$ of $\phi$ such that for any $\psi \in U$ we have   $J^{n}$ is a strong $(n,\epsilon,k)$-horseshoe of $\psi$. 
 
 \medskip
 
    $H(n,k)$ is dense in $C^{0}(N)$:  fix $\psi\in C ^{0}(N)$ with a $s$-periodic    point. Every small  neighborhood of the orbit of this point can be perturbed in order to obtain a strong $\left(n,\frac{1}{i^{2}},3^{n\, k\, i}\right)$   horseshoe for   a $\phi$ close to $ \psi$ for a large enough $i$ (see the proof of Theorem \ref{densitypositivemanifold}).
   
   \medskip 
   
     The above facts prove that   $   \mathcal{H}^{n}=\overset{\infty }{\underset{k=1}{\bigcap}} H(n,k) $  is residual in $C^{0}(N)$.

Finally, we prove    $\overline{\text{mdim}}_{\text{M}}(N,d,\phi)=n$ for any $\phi\in \mathcal{H}^{n}$.   Take $\phi\in \mathcal{H}^{n}$. We have $\phi \in H(n,k)$ for any $k\geq 1$. Therefore, for any $k\in\mathbb{N}$, there exists   $i_{k}$, with $i_{k}<i_{k+1}$,  such that $\phi $ has a strong  $\left(n, \frac{1}{i_{k}^{2}},3^{n\, k\, i_{k}} \right)$-horseshoe $J_{i_{k}}^{n}$, consisting of $3^{n\,k \,i_{k}}$ rectangular boxes $J^{n}(i_{k},1),\dots, J^{n}({i_{k}, 3^{n\,k\,i_{k}}})$, such that $J^{n}\subset (\phi_{i}(\hat{J}^{n}(i_{k},t)))^{\circ}$ for each $t=1,\dots, 3^{n\, k\, i_{k}}$, where $\phi_{i}=\text{exp}_{(i+1)\text{mod} \, s}\circ \phi\circ \text{exp}^{-1}_{i}$.   For each $k\in\mathbb{N}$, set $\varepsilon_{k}=\frac{1}{4i_{k}^{2}3^{k \,i_{k} }}$. For any $m\in\mathbb{N}$, set
$$  C_{n,k}(t_{0},t_{1},\dots, t_{m-1}) =\{x\in J^{n}_{i_{k}}: \phi^{l}(x)\in \hat{J}_{i_{k},t_{l}}^{n}\text{ for all }l\in\{0,\dots,m-1 \}\}.$$   From the definition, we have  $|J_{i_{k},t}^{n}|> \varepsilon_{k},$ for each $t=1,\dots, 3^{n\, k\, i_{k}}$.    Thus, the set consisting of any point on each    $ \tilde{C}_{n,k}=\text{exp}[ C_{n,k}(t_{0},t_{1},\dots, t_{m-1})]$  is an $(m ,\phi, \varepsilon_{k})$ separated set.  Therefore,  for each $m\in\mathbb{N}$ we have $$ \text{sep}(m, \phi,\varepsilon_{k})  \geq \left({3^{n\, k\, i_{k}}}  \right)^{m} \quad\text{and hence }\quad    \frac{\text{sep}( \phi,\varepsilon_{k})}{|\log \varepsilon_{k}|}\geq  n \frac{\log  {3^{k\, i_{k}}}  }{\log 3^{k \, i_{k} }+\log 4i_{k}^{2}} .$$
Note $   \frac{\log  {3^{k\,i_{k}}}  }{\log 3^{k\,i_{k}}+\log 4i_{k}^{2}}  \rightarrow 1$ as $k\rightarrow \infty$. This fact implies that 
$$   \overline{\text{mdim}}_{\text{M}}(N ,d,\phi)\geq n,$$
which proves the theorem, since for any $\psi \in C^{0}(N)$, the inequality $   \overline{\text{mdim}}_{\text{M}}(N ,d ,\psi)\leq n$  always hold.
\end{proof}

The continuity of the topological entropy is one of the most studied problem in dynamical systems (see \cite{block}, \cite{Sheldon}, \cite{Yano}).   If $X$ is the interval  or the circle,  Block, in \cite{block}, proved the topological entropy map is not continuous on continuous maps on $X$ with finite topological entropy.  Now,     Yano in \cite{Yano} proved   the topological entropy map is continuous on any continuous map $\phi \in  C^{0}(N)$  with infinite topological entropy.  For the metric mean dimension, it   follows from   Theorem \ref{densitypositivemanifold} that: 
 
\begin{corollary}\label{continuitymandimension}
If $N$ is any compact  riemannian manifold with riemannian metric $d$, then $\emph{mdim}_\emph{M}:C^{0}(N)\rightarrow \mathbb{R}$ is not continuous anywhere. 
\end{corollary}

 A real valued function $\varphi : X \rightarrow \mathbb{R}\cup \{\infty\}$ is called \textit{lower} (respectively \textit{upper})  \textit{semi-continuous on a point} $x\in X$ if $$\liminf_{y\rightarrow x}\varphi (y)\geq \varphi (x)\quad  (\text{respectively } \limsup_{y\rightarrow x}\varphi (y)\leq \varphi (x) ).    $$    $\varphi  $ is called \textit{lower} (respectively \textit{upper})  \textit{semi-continuous} if is  lower (respectively  {upper})  {semi-continuous on any point} of $ X$.

 \medskip
 
 The map $h_{top}: C^{0}([0,1])\rightarrow \mathbb{R}\cup \{\infty\}$ is lower semi-continuous (see  \cite{Misiurewicz}, Corollary 1). However, for metric mean dimension we have if $X=[0,1]$ or $\mathbb{S}^{1}$, then  $\text{mdim}_\text{M}:C^{0}(X)\rightarrow \mathbb{R}$ is nor lower neither upper  semi-continuous (see \cite{MeandFagner}, Proposition 7.6). {Furthermore,   from  Theorem \ref{densitypositivemanifold} we have}:
 
  \begin{corollary}\label{hfjdjehr}  Let  $N$ be any compact  riemannian manifold with riemannian metric $d$. We have  $\emph{mdim}_\emph{M}:C^{0}(N)\rightarrow \mathbb{R}$   is nor lower neither upper  semi-continuous on maps with metric mean dimension in $(0,\emph{dim} (N))$. Furthermore,   
 $\emph{mdim}_\emph{M}:C^{0}(N)\rightarrow \mathbb{R}$  is not  lower   semi-continuous on maps with metric mean dimension in $(0,\emph{dim} (N)]$ and is not  upper   semi-continuous on maps with metric mean dimension in $[0,\emph{dim} (N))$. 
 \end{corollary}

\section{Density of continuous maps on Cantor sets with positive metric  mean dimension}\label{section4}  
 
     Bobok and Zindulka in \cite{Bobok} shown that  if  $X$ is an uncountable compact metrizable space of topological dimension zero, then given any $a \in  [0, \infty]$ there exists a homeomorphism on $X$ whose topological entropy
is $a$. In particular, there exist homeomorphisms on the Cantor set with infinite topological entropy. We will use the techniques presented  by Bobok and Zindulka in order to prove  there exist infinitely many continuous maps  on the Cantor set with positive metric mean dimension. In fact, any $x \in [0, 1]$  is written in base 3 as
$$x=\sum _{n=1}^{\infty}x_{n}3^{-n}\quad\text{ where }x_{n}\in \{0,1,2\}.$$  A number $x$ belongs to the
ternary Cantor set if and only if it has a ternary representation where the digit one
does not appear.  Therefore, we can consider  \begin{equation}\label{nevsuidnf} \textit{\textbf{C}}=\{(x_{1},x_{2},\dots ): x_n=0,2\text{  for }n\in\mathbb{N}\}=\{0,2\}^{\mathbb{N}}\end{equation}
as being  the Cantor set endowed with the metric 
\begin{equation}\label{gsgdf} d((x_{1},x_{2},\dots),(y_{1},y_{2},\dots))=\sum_{n=1}^{\infty} 3^{-n}|x_{n}-y_{n}|=\left|\sum _{n=1}^{\infty}x_{n}3^{-n}-\sum _{n=1}^{\infty}y_{n}3^{-n}\right|. \end{equation}

\begin{proposition}\label{gshfkf} For each $j\in\mathbb{N}$, there exists $\psi_{j}\in C^{0}(\textit{\textbf{C}})$ with \begin{equation*}\underline{\emph{mdim}}_\emph{M}({\textit{\textbf{C}}}   ,d,\psi_{j})= \overline{\emph{mdim}}_\emph{M}({\textit{\textbf{C}}}   ,d,\psi_{j})= \frac{j\log 2}{(j+1)\log 3}.\end{equation*}
\end{proposition}
\begin{proof}For any $k\geq 1$,    set  \begin{align*}\textit{\textbf{C}}_{k}=\{(x_{i})_{i=1}^{\infty}: x_{i}=0 \text{ for }i\leq k-1, x_{k}=2\text{ and } x_{i}\in \{0,2\}\text{ for }i\geq k+1\}.\end{align*} Note that if $k\neq s$, then $ \textit{\textbf{C}}_{k}\cap \textit{\textbf{C}}_{s}=\emptyset $ and ${\textit{\textbf{C}}\setminus \cup_{k=1}^{\infty} \textit{\textbf{C}}_{k}}=\{(0,0,\dots)\}$.  Furthermore,   each  $\textit{\textbf{C}}_{k}$ is a clopen subset  homeomorphic to $\textit{\textbf{C}}$ via the  homeomorphism 
$$T_{k}:\textit{\textbf{C}}_{k}\rightarrow \textit{\textbf{C}},\quad (\underset{(k-1)\text{-times}}{\underbrace{{0},\dots, {0}}},2,x_{1},x_{2},\dots)\mapsto (x_{1},x_{2},\dots),$$ which is Lipschitz.     
For $j \in\mathbb{N}$, take $\psi_{j}:\textit{\textbf{C}}\rightarrow \textit{\textbf{C}}$ the map defined   as $\psi_{{j}}(0,0,\dots)=(0,0,\dots)$  and $  \psi_{j}|_{\textit{\textbf{C}}_{k}}= T_{k}^{-1}\sigma^{jk} T_{k} $ for $k\geq 1$.    It is not difficult to prove that $\psi_{{j}}$ is a continuous map.  
Take $\varepsilon >0$. For any $k\geq 1,$ set    $\varepsilon_{k}=3^{-k(j+1)}$. There exists $k\geq 1$ such that $\varepsilon\in[\varepsilon_{k+1},\varepsilon_{k}]$. 
For $n\geq 1$ and $k\geq 1$, take $\bar{z}_{1}=(z_{1}^{1},\dots,z_{jk}^{1}),\dots, \bar{z}_{n}=(z_{1}^{n},\dots,z_{jk}^{n})$, with $z_{i}^{s}\in \{0,2\},$ and set
$$ {A}^{k}_{\bar{z}_{1},\dots,\bar{z}_{n}}=\{(\underset{(k-1)\text{-times}}{\underbrace{{0},\dots, {0}}},2,z_{1}^{1},\dots,z_{jk}^{1},\dots,z_{1}^{n},\dots,z_{jk}^{n}, x_{1},\dots,x_s,.\,.\,.  ): x_{i}\in\{0,2\}\}\subseteq \textit{\textbf{C}}_{k}.$$
Note that if ${A}^{k}_{\bar{z}_{1},\dots,\bar{z}_{n}}\neq {A}^{k}_{\bar{w}_{1},\dots,\bar{w}_{n}}$ and $\bar{x}\in {A}^{k}_{\bar{z}_{1},\dots,\bar{z}_{n}}$, $\bar{y}\in {A}^{k}_{\bar{w}_{1},\dots,\bar{w}_{n}}$, then  $d_{n+1}(\bar{x},\bar{y})>\frac{1}{3^{k(j+1)}}$, where $d_{n+1}$ is considered with respect to $\psi_{j}$.   Therefore 
$  \text{sep}(n+1,\psi_{j},\varepsilon_{k}) \geq 2^{jnk}$ and hence 
\begin{align*}\limsup_{n\rightarrow \infty} \frac{\log\text{sep}(n+1,\psi_{j},\varepsilon) }{n+1}\geq 
\limsup_{n\rightarrow \infty} \frac{\log\text{sep}(n+1,\psi_{j},\varepsilon_{k}) }{n+1}& \geq \lim_{n\rightarrow \infty} \frac{n\log(2^{jk})}{n+1}= \log 2^{jk},\end{align*}
thus 
\begin{align*}  \underline{\text{mdim}}_\text{M}({\textit{\textbf{C}}} ,d,\psi_{j})& \geq
\lim_{k\rightarrow \infty} \frac{\log\text{sep}(\psi_{j},\varepsilon_{k}) }{-\log \varepsilon_{k+1}} \geq  \lim_{k\rightarrow \infty}\frac{\log(2^{jk})}{\log (3^{(k+1)(j+1)})} = \lim_{k\rightarrow \infty} \frac{{k}j\log2}{(k+1)(j+1)\log 3}\\
&=\frac{j\log 2}{(j+1)\log 3}.\end{align*}
Therefore  \begin{equation}\label{iuhugl}  \overline{\text{mdim}}_\text{M}({\textit{\textbf{C}}} ,d,\psi_{j})\geq \underline{\text{mdim}}_\text{M}({\textit{\textbf{C}}} ,d,\psi_{j})\geq  \frac{j\log 2}{(j+1)\log 3}.\end{equation}

On the other hand,   note that for each $l\in \{1,\dots,k\}$, the sets $ {A}^{l}_{\bar{z}_{1},\dots,\bar{z}_{n}}$ have $d_{n}$-diameter less than
 $\varepsilon_k$. Furthermore, the sets $\{(0,0,\dots)\}$ and  $\overset{\infty}{\underset{s=k+1}{\bigcup}}  \textbf{\textit{C}}_{s}$ has $d_{n}$-diameter less than
 $\varepsilon_k$. Hence $$\text{cov}(n , \psi_{j},\varepsilon_{k})\leq k2^{njk} +2 \leq2k2^{njk}$$ 
and therefore
\begin{equation*}\text{cov}( \psi_{j},\varepsilon_{k})\leq\lim_{n\rightarrow\infty} \frac{\log(2k2^{njk})}{n}= \log 2^{jk}. \end{equation*} Hence  \begin{equation}\label{iuhugl2}  \overline{\text{mdim}}_\text{M}({\textit{\textbf{C}}}   ,d,\psi_{j})=\limsup_{\varepsilon\rightarrow 0}\frac{\text{cov}(\psi_{j},\varepsilon)}{-\log\varepsilon}\leq \limsup_{k\rightarrow \infty}\frac{\text{cov}(\psi_{j},\varepsilon_{k+1})}{-\log\varepsilon_{k}} \leq \frac{j\log 2}{(j+1)\log 3}.\end{equation}
It follows from \eqref{iuhugl} and \eqref{iuhugl2} that \begin{equation*}\underline{\text{mdim}}_\text{M}({\textit{\textbf{C}}}  ,d ,\psi_{j})= \overline{\text{mdim}}_\text{M}({\textit{\textbf{C}}}   ,d,\psi_{j})= \frac{j\log 2}{(j+1)\log 3},\end{equation*}
which proves the proposition. 
\end{proof}

For any  continuous map  $\phi:X\rightarrow X$  we always have 
$$\overline{\text{mdim}}_{\text{M}}(X ,d,\phi) \leq \overline{\text{dim}}_{\text{B}} (X,d) \quad \text{ and }\quad \underline{\text{mdim}}_{\text{M}}(X ,d,\phi) \leq \underline{\text{dim}}_{\text{B}} (X,d) .$$  
Therefore  $ {\text{mdim}}_\text{M}({\textit{\textbf{C}}}   ,d,\phi)\leq \text{dim}_{B}(\textit{\textbf{C}})=\frac{\log 2}{\log 3}$ for any continuous map $\phi:\textit{\textbf{C}}\rightarrow \textit{\textbf{C}}$. A question that arises from the above proposition is:  ?`is there any $\phi \in C^{0}(\textit{\textbf{C}})$ with $ {\text{mdim}}_\text{M}({\textit{\textbf{C}}}   ,d,\phi) =\frac{\log 2}{\log 3}$?

\begin{remark}\label{obs54} Consider $\psi_{j}$ as in    Proposition \ref{gshfkf}. 
Note that $\psi _{sj}=\psi_{j}^{s}$  for any $s\in\mathbb{N}$.  It follows from the proposition that  \begin{equation*} \text{mdim}_\text{M}({\textit{\textbf{C}}}   ,d,\psi_{j}^{s})= \text{mdim}_\text{M}({\textit{\textbf{C}}}   ,d,\psi_{sj})=  \frac{sj\log 2}{(sj+1)\log 3} =\frac{s\log 2}{(s+\frac{1}{j})\log 3}.\end{equation*}
Therefore, for any $s\in\mathbb{N}$, we have \begin{equation*} \text{mdim}_\text{M}({\textit{\textbf{C}}}   ,d,\psi_{j})<\text{mdim}_\text{M}({\textit{\textbf{C}}}   ,d,\psi_{j}^{s})< s \, {\text{mdim}_\text{M}}({\textit{\textbf{C}}}   ,d,\psi_{j}).\end{equation*}
\end{remark}

For any $m\geq2$, take $X_{m}=\{1,2,\dots,m\}$. We endow $X_{m}^{\mathbb{K}}$ with the metric $d$ given in \ref{gsgdf}. It follows from   Proposition \ref{gshfkf}  there exist   continuous maps on $X_{m}^{\mathbb{K}}$ with positive metric mean dimension.

 \begin{theorem}\label{bcbcbcb123} Take $\mathbb{K}=\mathbb{N}$ or $\mathbb{Z}$.   If  $C_{a}=\{\phi\in C^{0}(X_{m}^{\mathbb{K}}):  \emph{mdim}_\emph{M}(X_{m}^{\mathbb{K}} ,  {d},\phi)=a\}\neq \emptyset$, then $C_{a}$   is dense in $C^{0}(X_{m}^{\mathbb{K}})$.  
\end{theorem}
\begin{proof}   We will prove the case $\mathbb{K}=\mathbb{N}$  (the case $\mathbb{K}=\mathbb{Z}$  is analogous). We will fix a continuous map $\phi:X_{m}^{\mathbb{N}}\rightarrow X_{m}^{\mathbb{N}}$, given by   $\phi(x_{1},x_{2},\dots)=(y_{1}(\bar{x}),y_{2}(\bar{x}),\dots)$, for any   $\bar{x}=(x_{1},x_{2},\dots)\in X_{m}^{\mathbb{N}}$. We will approximate $\phi$ by a sequence of continuous maps in $C_{a}.$ 

Firstly, we  prove that $C_{0}$ is dense in  $C^{0}(X_{m}^{\mathbb{N}})$. Consider the sequence of continuous maps on $X_{m}^{\mathbb{N}}$, $(\phi_{n})_{n\in\mathbb{N}}$, defined by \begin{equation}\label{nffne}   \phi_{n}(\bar{x})=(y_{1}(\bar{x}),y_{2}(\bar{x}),\dots,y_{n}(\bar{x}),x_{0},x_{0},\dots)\quad
\text{for any }n\in\mathbb{N}\text{ and some }x_{0}\in X_{m}.\end{equation}
Since the image of $\phi_{n}$ is a finite set, then we have $\text{mdim}_\text{M}(X_{m}^{\mathbb{N}} ,d,\phi_{n})=0$  for any $n\in\mathbb{N}$. Note that  $\phi_{n}$ converges uniformly to $\phi$ as $n\rightarrow \infty$.  This fact   proves  the set   $C_0$ is dense in $C^{0}(X_{m}^{\mathbb{N}})$.

Next, fix $a>0$ and suppose that $C_{a}\neq \emptyset$. Since $C_{0}$ is dense in $C^{0}(X_{m}^{\mathbb{N}})$, in order to prove that $C_{a}$ is dense in $C^{0}(X_{m}^{\mathbb{N}})$ we can prove that any map in $C_{0}$ can be approximate by a sequence of maps in $C_{a}$. Therefore, we can suppose that $\phi\in C_{0}$ is a map as the given in \eqref{nffne}, that is, for any $\bar{x}=(x_{1},x_{2},\dots)\in X_{m}^{\mathbb{N}}$,  \begin{equation*}  \phi(\bar{x})=(y_{1}(\bar{x}),y_{2}(\bar{x}),\dots,y_{K}(\bar{x}),z_{0},z_{0},\dots)\quad
\text{for some }K\in\mathbb{N}\text{ and some }z_{0}\in X_{m}.\end{equation*}   Suppose that $\psi_{a}\in C_{a}$  is given by     $$\psi_{a}(\bar{x})=(z_{1}(\bar{x}),z_{2}(\bar{x}),\dots)\quad \text{ for any }\bar{x}=(x_{1},x_{2},\dots)\in X_{m}^{\mathbb{N}}.$$   For each $n\geq K+1$, set $ \bar{x}^{n}=(x_{n+1},x_{n+2},\dots).$   Consider the sequence of continuous maps on $X_{m}^{\mathbb{N}}$, $(\phi_{n})_{n\geq K+1}$, where 
$$   \phi_{n}(\bar{x})=(y_{1}(\bar{x}),y_{2}(\bar{x}),\dots,y_{K}(\bar{x}), \underset{(n-K)\text{-times}}{\underbrace{z_{0},\dots, z_{0}}} ,z_{1}(\bar{x}^{n}),z_{2}(\bar{x}^{n}),\dots)
\text{ for any }n\geq K +1\text{ and  }\bar{x}\in X_{m}^{\mathbb{N}}.$$  
We have    $\phi_{n}$ converges uniformly to $\phi$ as $n\rightarrow \infty$.  Note that \begin{equation}\label{msn}\sum_{j=n}^{\infty} 3^{-j}|x_{j-n+1}-y_{j-n+1}|=3^{1-n}\sum_{j=n}^{\infty} 3^{-j+n-1}|x_{j-n+1}-y_{j-n+1}|= 3^{1-n}\sum_{j=1}^{\infty} 3^{-j}|x_{j}-y_{j}|.\end{equation}

Next, fix $n\in\mathbb{N}$ and take $\varepsilon >0$. For any  $p\geq n\in\mathbb{N}$,    let $A$ be a $(p,\psi_{a},\varepsilon)$-{separated} set.    Take 
$$ \tilde{A}=\{ (x_{i})_{i\in\mathbb{N}}: x_{j}=z_{0} \text{ for }1\leq j\leq n, (x_{n+i})_{i\in\mathbb{N}}\in A\} =\underset{n\text{-times}}{\underbrace{\{z_{0}\}\times\dots \times\{ z_{0}\}}}\times A  .$$ 
Note that if $(x_{i})_{i\in\mathbb{N}}$ and $(y_{i})_{i\in\mathbb{N}}$ are two different sequences in  $A$, from \eqref{msn} we have $$ d_{p}^{\phi_{n}}((z_0 ,\dots, z_{0}, x_{1}, x_{2},\dots), (z_0 ,\dots, z_{0}, y_{1}, y_{2},\dots)    ) \geq 3^{1-n}d_{p}^{\psi_{a}} ((x_{1},\dots),(y_{1},\dots) )\geq 3^{1-n} \varepsilon.$$
 Hence $\tilde{A}$ is a $(p,\phi_{n},3^{1-n}\varepsilon)$-{separated} set. Therefore  $ \text{sep}( \psi_{a}, \varepsilon  )\leq \text{sep}( \phi_{n}, 3^{1-n}\varepsilon  ) $  
and thus $$   \limsup_{\varepsilon\rightarrow 0}\frac{\text{sep}( \psi_{a}, \varepsilon  )}{|\log \varepsilon|} \leq \limsup_{\varepsilon\rightarrow 0}\frac{\text{sep}( \phi_{n}, 3^{1-n}\varepsilon  )}{|\log 3^{1-n}+\log \varepsilon|} = \limsup_{\varepsilon\rightarrow 0}\frac{\text{sep}( \phi_{n}, 3^{1-n}\varepsilon  )}{|\log   3^{1-n} \varepsilon|},$$ which proves that $\overline{\text{mdim}}_\text{M}(X_{m}^{\mathbb{N}} ,d,\psi_{a})\leq  \overline{\text{mdim}}_\text{M}(X_{m}^{\mathbb{N}} ,d,\phi_{n})$.

On the other hand,  note that 
$$\Omega(\phi_{n})\subseteq \underset{K\text{-times}}{\underbrace{X_{m}\times \cdots \times X_{m}}}\times \underset{(n-K)\text{-times}}{\underbrace{\{z_{0}\}\times\cdots \times\{z_{0}\}}}\times X_{m}\times X_{m}\times \cdots  := Z,$$ where  $\Omega(\varphi)$ is the non-wandering set of a continuous map $\varphi$. Hence, we can consider the restriction   
$\phi_{n}|_{Z}:Z\rightarrow Z$ in order to find the metric mean dimension of $\phi_{n}$.  Take $\varepsilon < 3^{-n}$ small enough such that if $d(\bar{x},\bar{y})<\varepsilon$, then $d(\phi (\bar{x}),\phi(\bar{y}))<3^{-K}$. Let $B$ be a 
   $(p,\psi_{a},\varepsilon)$-spanning set and  $C$ a  $(p,\phi,\varepsilon)$-spanning set.  Set $$\tilde{C}=\{ (x_{1},\dots, x_{K}): (x_{i})_{i\in\mathbb{N}}\in C \}\quad\text{and}\quad \tilde{B}= \tilde{C}\times\underset{(n-K)\text{-times}}{\underbrace{ \{z_{0}\}\times \cdots \times \{z_{0}\}}}\times B.   $$
 Take any $\bar{y}=(y_{1},y_{2},\dots,y_{K},z_{0},\dots, z_{0},y_{n+1},y_{n+2},\dots)\in Z$. There exists  $$\bar{a}=(y_{1},y_{2},\dots,y_{K},z_{0},\dots, z_{0}, a_{n+1}, a_{n+2},\dots)\in C$$  with $d^{\phi} _{p}(\bar{y},\bar{a})<\varepsilon $ 
 ($\bar{a}$ has this form because $\varepsilon<3^{-n}$). 
 Set   $$\tilde{x}=(y_{1},y_{2}, \dots, y_{K},z_{0},\dots,z_{0},x_{n+1},x_{n+2},\dots),$$ for some  $(x_{n+1},x_{n+2},\dots)\in B$ with $d_{p}^{\psi_{a}}((y_{n+1},y_{n+2},\dots), (x_{n+1},x_{n+2},\dots))<\varepsilon$. In particular    $d((y_{n+1},y_{n+2},\dots), (x_{n+1},x_{n+2},\dots))<\varepsilon$. Note that $\tilde{x}\in\tilde{B}$.  Hence, \begin{align*}d(\bar{y},\tilde{x})&=d((y_{1},\dots, y_{K}, z_{0},\dots, z_{0},y_{n+1},y_{n+2},\dots), (y_{1},\dots, y_{K}, z_{0},\dots, z_{0},x_{n+1},x_{n+2},\dots))\\
   &= \sum_{i=n+1}^{\infty}3^{-i}  |y_{i}-x_{i}| = 3^{-n}\sum_{i=n+1}^{\infty}3^{-i+n}  |y_{i}-x_{i}| =3^{-n}\sum_{i=1}^{\infty}3^{-i}  |y_{n+i}-x_{n+i}| \\ 
   & =3^{-n}d((y_{n+1},y_{n+2},\dots), (x_{n+1},x_{n+2},\dots)) <3^{-n}\varepsilon<\varepsilon  \end{align*}
   and therefore $$ d(\phi(\bar{y}),\phi(\tilde{x})) <3^{-K}.   $$ It follows from the definition of $\phi$ that $\phi(\bar{y})=\phi(\tilde{x}).$
  Thus
 \begin{align*}d_{p}^{\phi_{n}} (\bar{y}, \tilde{x}) & =   \max_{k=0,\dots, p-1}\left\{  \sum_{j=1}^{\infty}3^{-j}|(\phi_{n}^{k}(\bar{y}))_{j}- (\phi_{n}^{k}(\tilde{x}))_{j} |\right\}\\
 & =   \max_{k=0,\dots, p-1}\left\{ \sum_{j=1}^{n} 3^{-j} |(\phi_{n}^{k}(\bar{y}))_{j}- (\phi_{n}^{k}(\tilde{x}))_{j} | + \sum_{j=n+1}^{\infty}3^{-j}|(\phi_{n}^{k}(\bar{y}))_{j}- (\phi_{n}^{k}(\tilde{x}))_{j} |\right\}\\
 & =   \max_{k=0,\dots, p-1}\left\{ \sum_{j=1}^{n} 3^{-j} |(\phi^{k}(\bar{y}))_{j}- (\phi^{k}(\tilde{x}))_{j} | + \sum_{i=n+1}^{\infty}3^{-j}|(\phi_{n}^{k}(\bar{y}))_{j}- (\phi_{n}^{k}(\tilde{x}))_{j} |\right\}\\
 & =   \max_{k=0,\dots, p-1}\left\{3^{-n}\sum_{j=1}^{\infty}3^{-j}|(\psi_{a}^{k}(\bar{y}^{n}))_{j}- (\psi_{a}^{k}(\tilde{x}^{n}))_{j} |\right\}\\
 &=     3^{-n}d_{p}^{\psi_{a}}((y_{n+1},y_{n+2},\dots), (x_{n+1},x_{n+2},\dots)) <   3^{-n}\varepsilon.
 \end{align*}
 This fact proves $\tilde{B}$ is a $(p,\phi_{n},3^{-n}\varepsilon)$-{spanning} set. Hence  $$ \text{span}(p, \phi_{n}, 3^{-n}\varepsilon  )\leq \text{span}(p, \psi_{a},  \varepsilon  )\cdot \text{span}(p, \phi,  \varepsilon  ) $$  
and thus 
 $$ \text{span}(\phi_{n}, 3^{-n}\varepsilon  )\leq \text{span}(\psi_{a},  \varepsilon  )+ \text{span}( \phi,  \varepsilon  ) .$$  Therefore
\begin{align*}  \limsup_{\varepsilon\rightarrow 0}\frac{\text{span}( \phi_{n}, 3^{-n}\varepsilon  )}{|\log   3^{-n} \varepsilon|} &= \limsup_{\varepsilon\rightarrow 0}\frac{\text{span}( \phi_{n}, 3^{-n}\varepsilon  )}{|\log 3^{-n}+\log \varepsilon|} \\
&\leq \limsup_{\varepsilon\rightarrow 0}\frac{\text{span}( \psi_{a}, \varepsilon  )}{|\log \varepsilon|}+\limsup_{\varepsilon\rightarrow 0}\frac{\text{span}( \phi, \varepsilon  )}{|\log \varepsilon|} \\
&= \limsup_{\varepsilon\rightarrow 0}\frac{\text{span}( \psi_{a}, \varepsilon  )}{|\log \varepsilon|},\end{align*} which proves that  $  \overline{\text{mdim}}_\text{M}(X_{m}^{\mathbb{N}} ,d,\phi_{n})\leq \overline{\text{mdim}}_\text{M}(X_{m}^{\mathbb{N}} ,d,\psi_{a})$. 
Analogously we can prove that $\underline{\text{mdim}}_\text{M}(X_{m}^{\mathbb{N}} ,d,\psi_{a})=  \underline{\text{mdim}}_\text{M}(X_{m}^{\mathbb{N}} ,d,\phi_{n})$.   These facts proves the theorem. 
\end{proof}
 
\begin{remark} If $p= \frac{j\log 2}{(j+1)\log 3}$ for some $j\in\mathbb{N}$, it follows from Proposition \ref{gshfkf} and Theorem \ref{bcbcbcb123} that $C_{p} $ is dense in $C^{0}(X_{m}^{\mathbb{N}})$.\end{remark}
  
   Block, in \cite{block}, studied the continuity of the topological entropy map on the set consisting of continuous maps on the Cantor set, the interval and the circle. On the continuity  of the metric mean dimension on the set consisting of continuous maps on the product space $X_{m}^{\mathbb{K}}$ (in particular on  the Cantor set), we have from \eqref{nevsuidnf},    Proposition \ref{gshfkf} and Theorem \ref{bcbcbcb123} that: 

\begin{theorem}\label{bcbcbcb} If $m\geq 2$, then  $\emph{mdim}_\emph{M}:C^{0}(X_{m}^{\mathbb{K}})\rightarrow \mathbb{R}$ is not continuous anywhere. In particular,   $\emph{mdim}_\emph{M}:C^{0}(\textit{\textbf{C}})\rightarrow \mathbb{R}$ is not continuous anywhere. 
\end{theorem}

 It is well-known that any  perfect, compact, metrizable, zero-dimensional space  is homeomorphic to the middle third Cantor set (see   \cite{Engelking}, Chapter 6).  Hence, suppose that $X$ is a perfect, compact, metrizable, zero-dimensional space and let $\psi:X \rightarrow \textit{\textbf{C}}$ be an homeomorphism. Consider the metric on $X$ given by  $$d_{\psi}(x,y)=d(\psi(x),\psi(y)) \quad \text{ for }x,y\in X,$$ where $d$ is the metric given in \eqref{gsgdf}.    Note that if $\rho$ is other metric on $X$ which induces the same topology that $d_{\psi}$ on $X$, then $ \hat{\rho}(\phi,\varphi)=\underset{x\in X}{\max}\rho(\phi(x),\varphi(x))$,  for any $ \phi, \varphi \in   C^{0}(X)$, induces the same topology on $C^{0}(X)$ that the metric   $ \hat{d}_{\psi}(\phi,\varphi)=\underset{x\in X}{\max}d_{\psi}(\phi(x),\varphi(x))$. Therefore, the continuity of   $\text{mdim}_\text{M}:C^{0}(X)\rightarrow \mathbb{R}\cup\{\infty\}$ does  not depend on equivalent metrics on $X$.
It follows from Theorem \ref{bcbcbcb}   that:

\begin{corollary}\label{hd} Suppose that $X$ is a perfect, compact, metrizable, zero-dimensional space endowed with the metric $d_{\psi}$. The map
$\emph{mdim}_\emph{M}:C^{0}(X,d_{\psi})\rightarrow \mathbb{R}$ is not continuous anywhere. Therefore, for any   perfect, compact, metric, zero-dimensional space $(X,d)$, the map $\emph{mdim}_\emph{M}:C^{0}(X,d)\rightarrow \mathbb{R}$ is not continuous anywhere. 
\end{corollary}

Next, we will consider the map $\text{mdim}:C^{0}(X)\rightarrow \mathbb{R}\cup\{\infty\}$. Note if $X$ is a finite set, then  $\text{dim}(X^{\mathbb{K}})=0$. Therefore, $\text{mdim}:C^{0}(X^{\mathbb{K}})\rightarrow \mathbb{R}$ is  a constant map. More generally, if $(X_{i} )_{i\in\mathcal{J}}$ is a family of  compact Hausdorff spaces with $\text{dim}(X_{i})=0$ for each  $i\in\mathcal{J}$, then $\text{dim}(\prod_{i\in\mathcal{J}}X_{i})=0$.   Hence $\text{mdim}:C^{0}(\prod_{i\in\mathcal{J}}X_{i})\rightarrow \mathbb{R}$  is a constant map.  We will suppose that $X$   is an $n$ ($n\geq 1$) dimensional compact metric space, with metric $d$.   We endow $X^{\mathbb{K}}$ with the product topology, which is obtained from any metric equivalent to the metric  \begin{equation*}\label{gsgxdfwdw} \tilde{d}((x_{1},x_{2},\dots),(y_{1},y_{2},\dots))=\sum_{n=1}^{\infty} 3^{-n}d(x_{n},y_{n})\quad\text{for any }(x_{1},x_{2},\dots),(y_{1},y_{2},\dots)\in X^{\mathbb{N}}, \end{equation*} 
for $\mathbb{K}=\mathbb{N}$ and \begin{equation*}\label{gsgxdfwsssdw} \tilde{d}((\dots, x_{-1},x_{0},x_{1},\dots),(\dots, y_{-1},y_{0},y_{1},\dots))=\sum_{n\in\mathbb{Z}} 3^{-|n|}d(x_{n},y_{n}),\end{equation*} for any $(\dots, x_{-1},x_{0},x_{1},\dots),(\dots,y_{-1},y_{0},y_{1},\dots)\in X^{\mathbb{Z}}, $ 
for $\mathbb{K}=\mathbb{Z}.$

 \begin{theorem}\label{mendimen}  Take $\mathbb{K}=\mathbb{N}$ or $\mathbb{Z}$ and $X$ any finite dimensional compact metric space.   \begin{enumerate}[i.]
     \item The set consisting of continuous maps on $X^{\mathbb{K}}$ with zero   mean dimension is dense in $C^{0}(X^{\mathbb{K}})$. 
     \item If there exists $\psi_{a}\in C^{0}(X^{\mathbb{K}})$ with mean dimension equal to $a$, then the set consisting of continuous maps with mean dimension equal to $a$ is dense in $C^{0}(X^{\mathbb{K}})$.
     \item     $\emph{mdim} :C^{0}(X^{\mathbb{K}})\rightarrow \mathbb{R}\cup \{\infty\}$ is constant or is  not continuous anywhere.   
     \end{enumerate}
\end{theorem}
\begin{proof}  We consider  $\mathbb{K}=\mathbb{N}$. We will fix a continuous map $\phi:X^{\mathbb{N}}\rightarrow X^{\mathbb{N}}$, given by   $\phi(x_{1},x_{2},\dots)=(y_{1}(\bar{x}),y_{2}(\bar{x}),\dots)$, for any   $\bar{x}=(x_{1},x_{2},\dots)\in X^{\mathbb{N}}$.   Consider the sequence of continuous maps on $X^{\mathbb{N}}$, $(\phi_{n})_{n\in\mathbb{N}}$, defined by \begin{equation*}    \phi_{n}(\bar{x})=(y_{1}(\bar{x}),y_{2}(\bar{x}),\dots,y_{n}(\bar{x}),x_{0},x_{0},\dots)\quad
\text{for any }n\in\mathbb{N}\text{ and some }x_{0}\in X.\end{equation*}
Note that $\Omega(\phi_{n})\subseteq \underset{n\text{-times}}{\underbrace{X\times \cdots\times X}}\times  \{x_{0}\}\times\cdots$, and then $\Omega(\phi_{n})$ is  a finite  dimensional space. Hence  $\text{mdim}(\phi_{n}, X^{\mathbb{N}})= \text{mdim}(\phi_{n}|_{\Omega(\phi_{n})}, \Omega(\phi_{n}))=0$,  since  any continuous map on a finite dimensional space has mean dimension equal to zero.   Note that  $\phi_{n}$ converges uniformly to $\phi$ as $n\rightarrow \infty$.  This fact   proves i. 

\medskip

We prove ii.  Suppose  there exists $\psi_{a}\in C^{0}(X^{\mathbb{N}})$, which is given by     $$\psi_{a}(\bar{x})=(z_{1}(\bar{x}),z_{2}(\bar{x}),\dots)\quad \text{ for any }\bar{x}=(x_{1},x_{2},\dots)\in X^{\mathbb{N}},$$  and $\text{mdim}(\phi_{n}, X^{\mathbb{N}})=a>0.$  Fix   $\phi\in C^{0}(X^{\mathbb{N}})$, which, without loss of generality, we can suppose that       \begin{equation*}  \phi(\bar{x})=(y_{1}(\bar{x}),y_{2}(\bar{x}),\dots,y_{K}(\bar{x}),z_{0},z_{0},\dots)\quad
\text{for some }K\in\mathbb{N}\text{ and some }z_{0}\in X,\end{equation*}  for any  $\bar{x}=(x_{1},x_{2},\dots)\in X^{\mathbb{N}}$. 
For each $n\geq K+1$, if $\bar{x}=(x_{1},x_{2},\dots)$, set $$ \bar{x}^{n}=(x_{n+1},x_{n+2},\dots)\quad\text{and}\quad \bar{x}_{n}=(x_{1},x_{2},\dots,x_{n},z_{0},z_{0},\dots) .$$    Consider the sequence of continuous maps on $X^{\mathbb{N}}$, $(\phi_{n})_{n\geq K+1}$, where 
$$   \phi_{n}(\bar{x})=(y_{1}(\bar{x}_{n}),y_{2}(\bar{x}_{n}),\dots,y_{K}(\bar{x}_{n}), \underset{(n-K)\text{-times}}{\underbrace{z_{0},\dots, z_{0}}} ,z_{1}(\bar{x}^{n}),z_{2}(\bar{x}^{n}),\dots)
\text{ for }n\geq K +1\text{ and  }\bar{x}\in X^{\mathbb{N}}.$$  
We have    $\phi_{n}$ converges uniformly to $\phi$ as $n\rightarrow \infty$.   On the other hand,  note that 
$$\Omega(\phi_{n})\subseteq \underset{K\text{-times}}{\underbrace{X\times \cdots \times X}}\times \underset{(n-K)\text{-times}}{\underbrace{\{z_{0}\}\times\cdots \times\{z_{0}\}}}\times X\times X\times \cdots  := Z.$$   Hence, we can consider the restriction   
$\phi_{n}|_{Z}:Z\rightarrow Z$ in order to find the   mean dimension of $\phi_{n}$.   
Define $\mathcal{I}:X^{\mathbb{N}}\rightarrow X^{\mathbb{N}}\times X^{\mathbb{N}}$, defined by $\mathcal{I}(\bar{x})=(\bar{x}_{n},\bar{x}^{n})$, and  $ {\Phi}_{n}:X^{\mathbb{N}}\times X^{\mathbb{N}}\rightarrow X^{\mathbb{N}}\times X^{\mathbb{N}}$, defined by $ {\Phi}_{n}(\bar{x},\bar{y})=(\phi(\bar{x}_{n}),\psi_{a}(\bar{y}))$. We have 
$$\Phi_{n} (\mathcal{I}(\bar{x}))= \Phi_{n} (\bar{x}_{n},\bar{x}^{n}) = (\phi(\bar{x}_{n}),\psi_{a}(\bar{x}^{n}))  = \mathcal{I}(\phi_{n}(\bar{x})),$$
Hence, $$ \text{mdim}(X^{\mathbb{N}},\phi_{n})\leq \text{mdim}(X^{\mathbb{N}}\times X^{\mathbb{N}},\Phi_{n}) \leq  \text{mdim}(X^{\mathbb{N}},\phi)+ \text{mdim}(X^{\mathbb{N}},\psi_{a})= \text{mdim}(X^{\mathbb{N}},\psi_{a}). $$

On the other hand, we can refine each   open cover of $Z$ to one of the form $$ \mathcal{A}=    {{\mathcal{A}_{1}\times \cdots \times \mathcal{A}_{K}}}\times \underset{(n-K)\text{-times}}{\underbrace{\{z_{0}\}\times\cdots \times\{z_{0}\}}}\times \mathcal{A}_{n+1}\times \mathcal{A}_{n+2}\times \cdots  ,$$
 where $\mathcal{A}_{i}$ is an open cover of $X$ and, for some $J$, $\mathcal{A}_{i}=X$ for all $i\geq J$.   Set \begin{align*}\mathcal{B}&=\mathcal{A}_{n+1}\times \mathcal{A}_{n+2}\times \cdots\\
 \tilde{\mathcal{B}}&=X\times \cdots \times X\times \{z_{0}\}\times \cdots \times \{z_{0}\}\times \mathcal{B}\\
 \mathcal{B}_{0}^{m}(\psi_{a})&=\mathcal{B} \vee (\psi_{a}^{-1}(\mathcal{B}))\vee \dots \vee (\psi_{a}^{-m}(\mathcal{B})) \\ \mathcal{A}_{0}^{m}(\phi_{n})&= \mathcal{A} \vee (\phi_{n}^{-1}(\mathcal{A})) \vee \dots \vee (\phi_{n}^{-m}(\mathcal{A}))\\
 \mathcal{\tilde{B}}_{0}^{m}(\phi_{n})&= \mathcal{\tilde{B}} \vee (\phi_{n}^{-1}(\mathcal{\tilde{B}})) \vee \dots \vee (\phi_{n}^{-m}(\mathcal{\tilde{B}})) .\end{align*} Let $\pi: X^{\mathbb{N}}\rightarrow X^\mathbb{N}$, given by $\pi (x_{1},\dots, x_{n},x_{n+1},\dots)=( x_{n+1},x_{n+2},\dots)$. Note that $$ \pi(\mathcal{\tilde{B}}_{0}^{m}(\phi_{n})) \succ \mathcal{B}_{0}^{m}(\psi_{a})\quad\text{ and }\quad \mathcal{A}_{0}^{m}(\phi_{n})\succ \mathcal{\tilde{B}}_{0}^{m}(\phi_{n}) .$$ Hence \begin{equation*}\label{yert}\mathcal{D}(\mathcal{B}_{0}^{m}(\psi_{a}))\leq   \mathcal{D}(\pi(\mathcal{\tilde{B}}_{0}^{m}(\phi_{n})))\leq  \mathcal{D}(\mathcal{\tilde{B}}_{0}^{m}(\phi_{n}))\leq \mathcal{D}(\mathcal{A}_{0}^{m}(\phi_{n}))\quad\text{for each }m\in\mathbb{N}.\end{equation*}    
 Therefore \begin{align*} \lim_{m\rightarrow \infty}\frac{\mathcal{D}(\mathcal{B}_{0}^{m}(\psi_{a}))}{m+1}&\leq \lim_{m\rightarrow \infty}\frac{\mathcal{D}(\mathcal{A}_{0}^{m}(\phi_{n}))}{m+1},\end{align*}
 which proves that $$  \text{mdim}(\phi_{n},X^{\mathbb{N}})\geq  \text{mdim}(\psi_{a},X^{\mathbb{N}})\quad\text{for any }n\geq K+1 .  $$
 
 \medskip
 
Note that iii follows from i and ii.\end{proof}

\section*{Acknowledgement}

The   author thanks Professor Sergio Augusto Roma\~na Ibarra from Universidade Federal do Rio de Janeiro for useful discussions about  this work. Also to the reviewer of this manuscript for their comments and suggestions that helped to highly improve the quality and presentation of the results.

\end{document}